%% file: Hypercubes-final.tex
\documentclass{amsart}

\usepackage[ruled,vlined,linesnumbered]{algorithm2e}
\DontPrintSemicolon
\SetArgSty{textnormal}

\usepackage[normalem]{ulem}

\usepackage{latexsym} 
\usepackage{amsmath} 
\usepackage{epsfig}
\usepackage{amssymb}
\usepackage{enumerate}
\usepackage{color}
\usepackage{graphicx}
\usepackage{hyperref}

\newcommand{\Blue}{\textcolor{black}}

\usepackage[dvipsnames]{xcolor}
\newcommand{\Green}{\textcolor{black}}

\newtheorem{theorem}{Theorem}[section]

\newtheorem{lemma}[theorem]{Lemma}

\newtheorem{corollary}[theorem]{Corollary}
\newtheorem{proposition}[theorem]{Proposition}
\newtheorem{observation}[theorem]{Observation}

\newcommand{\cC}{{\mathcal C}}
\newcommand{\cF}{{\mathcal F}}

\newcommand{\cL}{{\mathcal L}}

\newcommand{\cN}{{\mathcal N}}
\newcommand{\cP}{{\mathcal P}}
\newcommand{\cR}{{\mathcal R}}
\newcommand{\cS}{{\mathcal S}}
\newcommand{\cT}{{\mathcal T}} 

\newcommand{\rSPR}{\mathrm{rSPR}}

\title[]{Hypercubes and \Blue{Hamilton} Cycles of Display Sets of Rooted Phylogenetic Networks}

\author{Janosch D\"ocker, Simone Linz, and Charles Semple} 

\thanks{We thank the New Zealand Marsden Fund for their financial support.}

\address{Department of Computer Science, University of T\"ubingen, Germany}
\curraddr{School of Computer Science, University of Auckland, Auckland, New Zealand}
\email{janosch.doecker@auckland.ac.nz}

\address{School of Computer Science, University of Auckland, Auckland, New Zealand}
\email{s.linz@auckland.ac.nz}

\address{School of Mathematics and Statistics, University of Canterbury, Christchurch, New Zealand}
\email{charles.semple@canterbury.ac.nz}


\date{\today}
 
\begin{document}

\begin{abstract}
In the context of reconstructing phylogenetic networks from a collection of phylogenetic trees, several characterisations and subsequently algorithms have been established to reconstruct a phylogenetic network that collectively embeds all trees in the input in some minimum way. For many instances however, the resulting  network also embeds additional phylogenetic trees that are not part of the input. However, little is known about these inferred trees. In this paper, we explore the relationships among all phylogenetic trees that are embedded in a given phylogenetic network. First, we investigate some combinatorial properties  of the collection $\cP$ of all rooted binary phylogenetic trees that are embedded in a rooted binary phylogenetic network $\cN$. To this end, we associated a particular graph $G$, which we call rSPR graph, with the elements in $\cP$ and show that, if $|\cP|=2^k$, where $k$ is the number of vertices with in-degree two in $\cN$, then $G$ has a \Blue{Hamilton} cycle. Second, by exploiting rSPR graphs and properties of hypercubes, we turn to the well-studied class of rooted binary level-$1$ networks and give necessary and sufficient conditions for when a set of rooted binary phylogenetic trees can be embedded in a level-$1$ network without inferring any additional trees. Lastly, we show how these conditions translate into a polynomial-time algorithm to reconstruct such a network if it exists.
\end{abstract}

\maketitle

\section{Introduction}
Phylogenetic networks, which are used to represent treelike and non-treelike ancestral relationships between a set of present-day species, generalise  phylogenetic trees by allowing for cycles in the underlying graph. In the case of rooted phylogenetic networks, vertices with in-degree at least two, called {\it reticulations}, represent non-treelike events such as hybridisation or lateral gene transfer that cannot be represented by a single rooted phylogenetic tree, whereas vertices with in-degree one represent treelike speciation events. Software to reconstruct phylogenetic networks frequently uses molecular sequence data or a collection of conflicting phylogenetic trees as input~\cite{allman19,barrat22, huson12, iersel22, solis16}. If a phylogenetic network $\cN$ is reconstructed from a set of phylogenetic trees such that $\cN$ embeds each tree of the input, then it may be necessary for $\cN$ to not only embed the input, but also a number of additional phylogenetic trees. For example, referring to the two rooted phylogenetic trees $\cT_1$ and $\cT_2$ that are shown in Figure~\ref{fig:honeycomb}, every rooted phylogenetic network that embeds $\cT_1$ and $\cT_2$ has at least two reticulations and embeds at least one tree that is distinct from $\cT_1$ and $\cT_2$.  \Blue{In Figure~\ref{fig:honeycomb} and, in fact, in all figures of this paper, arcs of rooted phylogenetic networks are directed down the page and arrowheads are omitted.} Clearly, if the input consists of all rooted binary phylogenetic trees on a fixed leaf set, then no rooted phylogenetic network that embeds each tree in the input infers any additional tree. Such networks are called universal networks and exist, for example, for the class of tree-based networks~\cite{francis15,zhang16}. However, for more structurally restricted network classes such as level-$1$ or tree-child networks whose number of reticulations is bounded linearly in the number of leaves~\cite{mcdiarmid15}, no universal network exist. Moreover, in practice, one is often interested in a  subset of all rooted binary phylogenetic trees on a fixed leaf set. It is consequently more realistic to ask which collections of phylogenetic trees can be embedded in a network without inferring any additional tree? In this paper, we approach this  question from two angles.

First, we investigate the relationships among all rooted binary phylogenetic trees that are embedded in a given rooted binary phylogenetic network $\cN$. We refer to the set of all such trees as the {\it display set} (formally defined in the next section) of $\cN$. It is well-known that the size of the display set of a rooted binary phylogenetic network with exactly $k$ reticulations is at most $2^k$ and that this bound is sharp, such as for normal networks~\cite{iersel10,willson12}. However, not all rooted phylogenetic networks with $k$ reticulations have a display set of size $2^k$. An example is shown in Figure~\ref{fig:honeycomb}. To explore the relationships among the elements of a display set $\cP$ of a rooted phylogenetic network, we associate an undirected graph with $\cP$.  Referring to this graph as the {\it rSPR graph} of $\cP$, its vertex set is $\cP$ and two vertices are connected by an edge precisely if they are one rooted subtree prune and regraft (rSPR) operation~\cite{hein96} apart.~The rSPR operation induces a metric on the space of all rooted binary phylogenetic trees with a fixed leaf set. It is used to compare pairs of phylogenetic trees and to search for an optimum tree in tree space~\cite{allen01,bordewich05,goloboff08,stjohn17}. We show that the rSPR graph of a display set of a rooted binary phylogenetic network is always connected. In turn this implies that, if the rSPR graph of an arbitrary collection of rooted binary phylogenetic trees is not connected, then any rooted phylogenetic network that embeds each tree in the collection infers additional phylogenetic trees. Moreover, if $\cP$ is the display set of a rooted binary phylogenetic network with $k$ reticulations and has size $2^k$, then its rSPR graph $G$ has a \Blue{Hamilton} cycle. Hence,  in the spirit of~\cite{gordon13}, it is possible to systematically traverse $G$, thereby visiting each element in $\cP$ exactly once.

Second, we turn to level-$1$ networks that are phylogenetic networks whose underlying cycles do not intersect. We characterise when a set $\cP$ of rooted binary phylogenetic trees is the display set of a rooted binary level-$1$ network, in which case it is possible to reconstruct such a network that embeds each tree in $\cP$ and does not infer any additional tree. Our characterisation again employs rSPR graphs and establishes necessary and sufficient conditions for $\cP$  to be the display set of  a rooted binary level-$1$ network. To provide a flavour of the characterisation, a necessary condition is that the rSPR graph of $\cP$ is isomorphic to the $k$-dimensional hypercube for some non-negative integer $k$. Although hypercubes have previously been used in research on phylogenetic trees (e.g. in the context of Buneman graphs and maximum parsimony~\cite[Section 5.5]{semple03} as well as in developing a  lower bound on the minimum number of reticulations needed to explain a collection of conflicting phylogenetic trees~\cite{wu10}), their application to studying the display set of a phylogenetic network is new to this paper. Subsequent to the characterisation, we present a polynomial-time algorithm to decide whether or not $\cP$ is the display set of a rooted binary level-$1$ network and, if so, to reconstruct such a network. This algorithm can also be easily modified to enumerate all rooted binary level-$1$ network whose display set is $\cP$. Relatedly, there exists earlier work~\cite{huynh05,simpson} on reconstructing a (single) rooted binary level-$1$ network whose number of reticulations is minimised over all rooted binary level-$1$ networks whose display set is a superset of $\cP$, where the focus in~\cite{huynh05} is on the case $|\cP|=2$. Although these earlier algorithms can potentially be exploited further to also decide if there exists a rooted binary level-$1$ network whose display set is $\cP$, the purpose of the present paper is to demonstrate the applicability of rSPR graphs to studying display sets of rooted binary phylogenetic networks. Indeed, we expect that rSPR graphs will be used in the future to investigate related questions that go beyond level-$1$ networks.

The remainder of the paper is organised as follows. Section~\ref{sec:prelim} provides definitions and terminology that is used in the subsequent sections. We then establish basic properties of rSPR graphs in Section~\ref{sec:properties} and, in particular, hamiltonicity of any rSPR graph with $2^k$ vertices whose underlying collection of rooted binary phylogenetic trees is the display set of a rooted phylogenetic network with $k$ reticulations in Section~\ref{sec:maximum}. In Section~\ref{sec:level-1}, we use rSPR graphs to characterise when a set $\cP$ of rooted binary phylogenetic trees is the display set of a rooted level-$1$ network. Lastly, in Section~\ref{sec:level-1-algo}, we show that it takes polynomial time to decide if the necessary and sufficient conditions established in Section~\ref{sec:level-1} are satisfied and, if so, to reconstruct a level-$1$ network whose display set is $\cP$.

\begin{figure}[t]
\center
\scalebox{1}{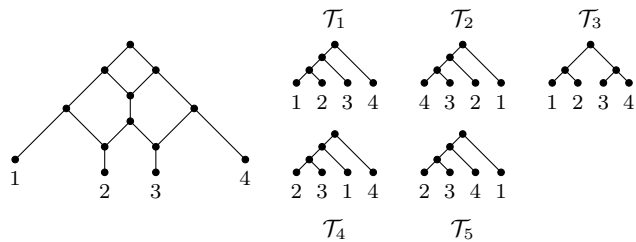}
\caption{A phylogenetic network with three reticulations whose display set consists of the five phylogenetic trees shown on the right.}
\label{fig:honeycomb}
\end{figure}

\section{Preliminaries}\label{sec:prelim}

This section gives definitions and terminology on phylogenetic trees and networks as well as on hypercubes that is used in the following sections. Throughout this paper, $X$ denotes a non-empty finite set.

\noindent {\bf Phylogenetic networks.} A {\em rooted binary phylogenetic network $\cN$ on $X$} is a rooted acyclic directed graph with no parallel arcs or loops that satisfies the following three properties:
\begin{enumerate}[(i)]
\item the (unique) root has in-degree zero and out-degree two;
\item a vertex of out-degree zero has in-degree one, and the set of vertices with out-degree zero is $X$; and
\item all other vertices either have in-degree one and out-degree two, or in-degree two and out-degree one.
\end{enumerate}
For technical reasons, if $|X|=1$, then we additionally allow $\cN$ to consist of the single vertex in $X$. Let $v$ be a vertex of $\cN$. If $v$ has out-degree zero, then $v$ is  called a {\em leaf}, and  $X$ is referred to as the {\em leaf set} of $\cN$. Furthermore, if $v$ has in-degree one and out-degree two, $v$ is referred to as a {\em tree vertex}, and if it has in-degree two and out-degree one, $v$ is referred to as a {\em reticulation}. For an arc $(u,v)$ of $\cN$, we say that $u$ is a {\it parent} of $v$ and, equivalently, that $v$ is a {\it child} of $u$. Also, if $v$ is a reticulation, then $(u,v)$ is called a {\it reticulation arc}. Lastly, if $u$ is a vertex of $\cN$, then $C_{\cN}(u)$ denotes the subset of $X$ whose elements $x$ have the property that there is a directed path from $u$ to $x$ in $\cN$. Such a subset of $X$ is referred to as a {\em cluster} of $\cN$ and we denote the set of clusters of $\cN$ by $C(\cN)$. Note that each element in $X$ is a cluster of $\cN$. If there is no ambiguity, we sometimes refer to $C_{\cN}(u)$ as $C(u)$.

We next consider different classes of phylogenetic networks that are well known in the literature. For an excellent overview on the different classes, we refer the interested reader to Kong et al.~\cite{kong} (and references therein). Let $\cN$ be a rooted binary phylogenetic network on $X$, and let $e=(u,v)$ be a reticulation arc of $\cN$. Then $e$ is called a {\it shortcut} of $\cN$ if there exists a directed path from $u$ to $v$ that avoids $e$. Now, if each non-leaf vertex of $\cN$ has a child that is a tree vertex or leaf, then $\cN$ is a {\it tree-child} network. Moreover, if $\cN$ is tree-child and does not contain a shortcut, then $\cN$ is a {\it normal} network. With a view towards the underlying cycles of $\cN$, we say that $\cN$ is a {\it level-$1$} network if no two underlying cycles of $\cN$ have a common vertex. Lastly, a {\em rooted binary phylogenetic $X$-tree} is a rooted binary phylogenetic network on $X$ with no reticulations. In what follows, we will refer to a rooted binary phylogenetic network and a rooted binary phylogenetic tree as a {\it phylogenetic network} and a {\it phylogenetic tree}, respectively, since all such networks and trees in this paper are rooted and binary.

We next define three types of subtrees of a phylogenetic $X$-tree $\cT$. Let $V$ be a subset of the vertices of $\cT$. First, we write $\cT(V)$ to denote the minimal rooted subtree of $\cT$ that connects all elements in $V$. Second, the {\it restriction of $\cT$ to $V$}, denoted by $\cT|V$, is the rooted phylogenetic tree obtained from $\cT(V)$ by suppressing each vertex  with in-degree one and out-degree one. In what follows,  $V$ is typically a subset of $X$. Third, a subtree of $\cT$ is called {\it pendant} if it can be detached from $\cT$ by deleting a single arc.

Now, let $\cN$ be a phylogenetic network on $X$ with $k$ reticulations, and let $\cT$ be a phylogenetic $X$-tree. We say that $\cT$ is {\it displayed} by $\cN$ if  there exists a subgraph of $\cN$ that is a subdivision of $\cT$. For a set $\cP$ of phylogenetic $X$-trees, we say that $\cN$ {\it displays} $\cP$ if each element in $\cP$ is displayed by $\cN$. Moreover the set of all phylogenetic $X$-trees that are displayed by $\cN$ is called the {\it display set} of $\cN$ and denoted by $T(\cN)$. Since the in-degree of each reticulation is two, it immediately follows that $|T(\cN)|\leq 2^k$. Furthermore, we say that $T(\cN)$ is {\it maximum} if $|T(\cN)|=2^k$. The class of phylogenetic networks whose display set is maximum strictly contains the class of normal networks~\cite{iersel10,willson12}. However, not every phylogenetic network $\cN$ with $k$ reticulations has a display set of size $2^k$. An example is shown in Figure~\ref{fig:honeycomb}. Moreover, a sufficient but not necessary condition for a phylogenetic network with $k$ reticulations to display strictly less than $2^k$ phylogenetic trees is the existence of an arc that is incident with two reticulations.

Let $\cN$ and $\cN'$ be two phylogenetic networks on $X$ with vertex and arc sets $V$ and $E$, and $V'$ and $E'$, respectively. Then $\cN$ and $\cN'$ are {\it isomorphic} if there is a bijection $\psi :V\rightarrow V'$ such that $\psi(x)=x$ for all $x\in X$, and $(u,v) \in E$ if and only if $(\psi(u),\psi(v)) \in E'$ for all $u,v\in V$. If $\cN$ and $\cN'$ are isomorphic, we write $\cN\cong \cN'$ and, otherwise, we write $\cN\ncong \cN'$.

\noindent {\bf rSPR and agreement forests.} Let $\cT$ be a phylogenetic $X$-tree. For the purposes of the upcoming definitions, we view the root of $\cT$ as a vertex $\rho$ adjoined to the original root by a pendant arc. Furthermore, we regard $\rho$ as part of the label set of $\cT$, that is, $\cL(\cT )=X\cup\{\rho\}$. Let $e=(u,v)$ be an arc of $\cT$ that is not incident with $\rho$. Let $\cT'$ be a phylogenetic $X$-tree obtained from $\cT$ by deleting $e$ and reattaching the resulting rooted subtree that contains $v$ via a new arc $f$ in the following way: Subdivide an arc of the component that contains $\rho$ with a new vertex $u'$, join $u'$ and $v$ with $f$, and suppress $u$. We say that $\cT'$ has been obtained from $\cT$ by a {\it rooted subtree prune and regraft} (rSPR) operation. The {\it rSPR distance}  between any two phylogenetic $X$-trees $\cT$ and $\cT'$, denoted by $d_\rSPR(\cT,\cT')$, is the minimum number of $\rSPR$ operations that transform $\cT$ into $\cT'$. It is well known that  $\cT'$ can always be obtained from $\cT$ by a sequence of single rSPR operations and, so, the distance is well defined.

Now, let $\cT$ and $\cT'$ be two phylogenetic $X$-trees.  An {\it agreement forest} $\cF=\{\cL_\rho,\cL_1,\ldots,\cL_k\}$ for $\cT$ and $\cT'$ is a partition of $X\cup\{\rho\}$ such that $\rho\in \cL_\rho$ and the following two properties are satisfied.

\begin{enumerate}[(i)]
\item  For each $i\in \{\rho,1,2,\ldots,k\}$, we have $\cT|\cL_i\cong \cT'|\cL_i$.
\item The trees in $\{\cT(\cL_i): i\in \{\rho,1,2,\ldots,k\}\}$ and 
$\{\cT'(\cL_i): i\in \{\rho,1,2,\ldots,k\}\}$ are vertex-disjoint subtrees of $\cT$ and $\cT'$, respectively. 
\end{enumerate}
An agreement forest for $\cT$ and $\cT'$ is a {\em maximum agreement forest} if it has the smallest number of elements amongst all agreement forests for $\cT$ and $\cT'$. The following theorem links the rSPR distance between two phylogenetic trees and the size of a maximum agreement forest for the same trees. 

\begin{theorem}\cite{bordewich05}
Let $\cT$ and $\cT'$ be two rooted phylogenetic $X$-trees, and let $\cF$ be a maximum agreement forest for $\cT$ and $\cT'$. Then $d_\rSPR(\cT,\cT')=|\cF|-1$.
\end{theorem}

Let $\cP$ be a set of phylogenetic $X$-trees. The {\it rSPR graph of $\cP$} is the graph $G=(V,E)$ with $V=\cP$ and for which $\{\cT,\cT'\}$ is an edge in $E$ precisely if $d_\rSPR(\cT,\cT')=1$.
Let $\cN$ be a phylogenetic network. For ease of reading, we often refer to the rSPR graph of $T(\cN)$ as the {\it rSPR graph of $\cN$}. To illustrate, Figure~\ref{fig:rSPR-graph} shows two phylogenetic networks with their rSPR graphs.

\begin{figure}[t]
\center
\scalebox{1}{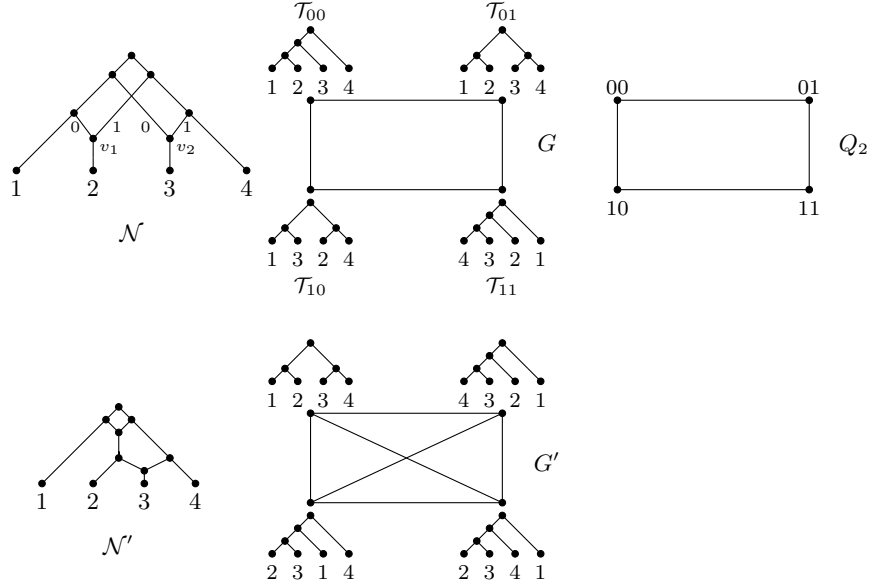}
\caption{Two phylogenetic networks $\cN$ and $\cN'$ with their rSPR graphs $G$ and $G'$, respectively. For each $2$-bit string $s$,  $\cT_s$ denotes the phylogenetic tree in $T(\cN)$ that is encoded by $s$  under the ordered binary assignment for $\cN$ as indicated by $v_1$, $v_2$, and the assignment of 0 or 1 to each reticulation arc of $\cN$.}
\label{fig:rSPR-graph}
\end{figure}

\noindent {\bf Gray codes and hypercubes.} Let $k$ be a non-negative integer, and let $n=2^k$. We refer to a string $s$ as a {\it $k$-bit string} if $s$ has length $k$ and each bit of $s$ is either $0$ or $1$. For two $k$-bit strings $s$ and $s'$, we denote the Hamming distance between $s$ and $s'$ by $d(s,s')$. Furthermore, an ordering $(s_1,s_2,\ldots,s_{n})$ on all $k$-bit strings is called a {\it (cyclic) Gray code} if $d(s_1,s_{n})=1$ and, for each $j\in\{1,2,\ldots,n-1)$, $d(s_j,s_{j+1})=1$~\cite{gray53}. It is well known that \Blue{such an ordering on $(s_1,s_2,\ldots,s_{n})$ exists (see, for example,~\cite{grimaldi,muetze22})}.

Let $k$ be a non-negative integer, and let $B$ be the set of all $k$-bit strings. If $k=0$, then the only element in $B$ is the empty string. The {\it $k$-dimensional hypercube $Q_k$} is the undirected graph whose vertex set is $B$ and for which $\{s,s'\}$ is an edge in $Q_k$ precisely if $d(s,s')=1$. Observe that the number of edges in $Q_k$ is $2^{k-1}k$. Moreover, the edge set of $Q_k$ can naturally be partitioned into $k$ sets $E_1,E_2,\ldots,E_k$ such that, for each $i\in\{1,2,\ldots,k\}$, we have $|E_i|=2^{k-1}$ and each edge $\{s,s'\}$ of $Q_k$ is an element of $E_i$ if and only if the $i$-th bit of $s$ and the $i$-th bit of $s'$ are not the same. We refer to $E_i$ as the $i$-th {\it bit edge subset} of $Q_k$. By way of example, $Q_2$ is shown in Figure~\ref{fig:rSPR-graph}, where $E_1$ contains the two vertical edges and $E_2$ contains the two horizontal edges.

We end this section with a well-known theorem whose proof is straightforward~\cite{muetze22}, and that establishes an equivalence between finding a Gray code for all $k$-bit strings and finding a \Blue{Hamilton} cycle of $Q_k$.

\begin{theorem}\label{t:code-iff-cube}
Let $k$ be an integer with $k\geq 2$, and let $n=2^k$. Furthermore, let $C=(s_1,s_2,\ldots,s_n)$ be an ordering on all $k$-bit strings. Then $C$ is a Gray code if and only if $$\{\{s_1,s_2\},\{s_2,s_3\},\ldots,\{s_{n-1},s_n\},\{s_n,s_1\}\}$$ is the edge set of a \Blue{Hamilton} cycle of $Q_k$.
\end{theorem}

\section{Properties of  rSPR Graphs}\label{sec:properties}

Let $\cN$ be a phylogenetic network, and let $R=\{v_1,v_2,\ldots,v_k\}$ be the set of reticulations in $\cN$. For each $i\in\{1,2,\ldots,k\}$, let $u_i$ and $u_i'$ be the two parents of $v_i$ in $\cN$. Furthermore, let $$\phi: \{(u_i,v_i),(u_i',v_i): i\in\{1,2,\ldots,k\}\}\rightarrow\{0,1\}$$ be a map that assigns either $0$ or $1$ to each reticulation arc such that $$\{\phi((u_i,v_i)),\phi((u_i',v_i))\}=\{0,1\}$$ for each $v_i\in R$.  We refer to $\phi$ as a {\it binary assignment} for $\cN$. Moreover, $(u_i,v_i)$ is called the {\it $1$-arc of $v_i$ under $\phi$} if $\phi(u_i)=1$ and, otherwise, $(u_i,v_i)$ is called the {\it $0$-arc of $v_i$ under $\phi$}. This definition extends in the obvious way to $(u_i',v_i)$. 

Now, let $\cN$ be a phylogenetic network on $X$, and let $\phi$ be a binary assignment for $\cN$. Let $R$ be the set of reticulations in $\cN$, and let $|R|=k$. Fix an ordering on the elements in $R$, say $(v_1,v_2,\ldots,v_k)$. Let $s$ be a $k$-bit string, and let $S$ be the directed spanning tree of $\cN$ such that, for each $i\in\{1,2,\ldots,k\}$, $S$ uses the $1$-arc of $v_i$ under $\phi$ if the $i$-th bit of $s$ is $1$ and $S$ uses the $0$-arc of $v_i$ under $\phi$ if the $i$-th bit of $s$ is $0$. Furthermore, let $\cT_s$ denote the phylogenetic $X$-tree that is obtained from $S$ by repeatedly suppressing vertices of in-degree one and out-degree one, deleting vertices with out-degree zero that are not in $X$, and deleting vertices with in-degree zero and out-degree one. Note that the last operation of deleting a vertex with in-degree zero and out-degree one is, for example, necessary for $\cT_s$ to be a phylogenetic tree if $\cN$ has an underlying $3$-cycle that contains the root and $S$ contains the unique reticulation arc of the $3$-cycle that is not incident with the root. We say that {\it $s$ encodes $\cT_s$ under $\phi$}. By construction, $\cT_s\in T(\cN)$. Each $k$-bit string encodes a unique element in $T(\cN)$ under $\phi$. Moreover two distinct $k$-bit strings may encode the same element in $\cT(\cN)$ under $\phi$. To illustrate, Figure~\ref{fig:rSPR-graph} shows, for each  $2$-bit string $s$,  the phylogenetic tree $\cT_s$ that is encoded under the binary assignment as indicated for the phylogenetic network $\cN$ shown in the same figure.

\noindent {\bf Notational remark.} Let $\cN$ be a phylogenetic network. Throughout this section and the next, we denote a binary assignment $\phi$  of $\cN$ and a fixed ordering $(v_1,v_2,\ldots,v_k)$ on the reticulations of $\cN$ with $k\geq 0$ by $$(\cN,\phi,(v_1,v_2,\ldots,v_k)).$$ Furthermore, we refer to $(\cN,\phi,(v_1,v_2,\ldots,v_k))$ as an {\it ordered binary assignment} of $\cN$.

\begin{lemma}\label{l:hamming1}
Let $\cN$ be a phylogenetic network on $X$, and let $s$ and $s'$ be two $k$-bit strings such that $d(s,s')=1$. Furthermore, let $\cT_s$ and $\cT_{s'}$ be the two phylogenetic $X$-trees that are encoded by $s$ and $s'$ under $\phi$, respectively. Then $d_\rSPR(\cT_s,\cT_{s'})\leq 1$. 
\end{lemma}

\begin{proof}
Let $k$ be the number of reticulations in $\cN$. Throughout this proof, let $(\cN,\phi,(v_1,v_2,\ldots,v_k))$ be an ordered binary assignment of $\cN$. If $\cT_s\cong\cT_{s'}$, then the result clearly follows. We may therefore assume that $\cT_s\ncong\cT_{s'}$. Let $S$ (resp. $S'$) be \Blue{the} directed spanning tree of $\cN$ such that, for each $i\in\{1,2,\ldots,k\}$, $S$ (resp. $S'$) uses the $1$-arc of $v_i$ under $\phi$ if the $i$-th bit of $s$ (resp. $s'$) is 1 and, otherwise, $S$ (resp. $S'$) uses the $0$-arc of $v_i$ under $\phi$. As $d(s,s')=1$, there exists exactly one reticulation $v_j$ in $\cN$ such that one of $S$ and $S'$ uses the $1$-arc of $v_j$ under $\phi$ while the other uses the $0$-arc of $v_j$ under $\phi$. Now, obtain a directed acyclic graph $\cN'$ from $\cN$ by deleting each arc that is directed into a reticulation and not used by $S$ or $S'$, and subsequently, applying any of the following three operations until no further operation is possible.
\begin{enumerate}[(i)]
\item Suppress a vertex of in-degree one and out-degree one.
\item Delete a vertex with out-degree zero that is not in $X$.
\item Delete a vertex of in-degree zero and out-degree one. 
\end{enumerate}
By construction, $v_j$ is the only vertex of $\cN'$ with in-degree 2. Moreover, since $\cT_s\ncong\cT_{s'}$, the two arcs that are directed into $v_j$ are not in parallel. Hence $\cN'$ is a phylogenetic network. Furthermore, as $\cT_s,\cT_{s'}\in T(\cN)$, we also have  $\cT_s,\cT_{s'}\in T(\cN')$. Since $\cT_s\ncong\cT_{s'}$ and, consequently, each phylogenetic network that displays $\cT_s$ and $\cT_{s'}$ has at least one reticulation, it now follows from~\cite[Proposition 2]{baroni05} that  $d_\rSPR(\cT_s,\cT_{s'})=1$. This completes the proof of the lemma.
\end{proof}

The next lemma shows that the rSPR graph of a phylogenetic network is always connected. 

\begin{lemma}\label{l:1-connected}
Let $\cP$ be a set of phylogenetic $X$-trees. If there exists a phylogenetic network $\cN$ with $T(\cN)=\cP$, then the rSPR graph of $\cP$ is connected.
\end{lemma}

\begin{proof}
Suppose that $\cP$ is the display set of a phylogenetic network $\cN$ on $X$ that has $k$ reticulations. Let $(\cN,\phi,(v_1,v_2,\ldots,v_k))$ be an ordered binary assignment of $\cN$. Furthermore, let $n=2^k$, and let $B$ be the set of all $k$-bit strings. Consider an ordering, say $(s_1,s_2,\ldots,s_n)$, on the elements in $B$ that is a Gray code. Now, for each $j\in\{1,2,\ldots, n\}$, let $\cT_j$ be the phylogenetic $X$-tree that is encoded by $s_j$ under $\phi$, and let  $S_j$ be the directed spanning tree of $\cN$ that, for each $i\in\{1,2,\ldots,k\}$, uses the $1$-arc of $v_i$ under $\phi$ if the $i$-th bit of $s_j$ is 1 and, otherwise, $S_j$ uses the $0$-arc of $v_i$ under $\phi$. Note that we may have $\cT_j=\cT_{j'}$ for two distinct elements $j,j'\in\{1,2,\ldots,n\}$. Let $G$ be the  undirected graph without loops whose vertex set is $\{S_1,S_2,\ldots,S_n\}$ and for which $\{S_j,S_{j'}\}$ is an edge precisely if $d_\rSPR(\cT_j,\cT_{j'})\leq 1$. By Lemma~\ref{l:hamming1}, it follows that each of $$\{S_1,S_2\},\{S_2,S_3\},\ldots,\{S_{n-1},S_n\}$$ is an edge in $G$ and, hence, $G$ is connected. If $|T(\cN)|=n$, then, as the elements $\cT_1,\cT_2,\ldots,\cT_n$ are pairwise distinct, it is straightforward to check that $G$ is isomorphic to the rSPR graph of $\cN$. Assume that $\cT_j\cong\cT_{j'}$ for two distinct elements $j,j'\in\{1,2,\ldots,n\}$. Since $d_\rSPR(\cT_j,\cT_{j'})=0$, the edge $\{S_j,S_{j'}\}$ exists in $G$. Moreover, $\{S_j,S_l\}$ is an edge in $G$ if and only if $\{S_{j'},S_l\}$ is an edge in $G$ with $l\in\{1,2,\ldots,n\}$. It now follows that the undirected graph obtained from $G$ by deleting $S_{j'}$ is connected. By construction, repeating this vertex deletion operation until there exists no further pair of vertices $S_j$ and $S_{j'}$ with $\cT_j\cong \cT_{j'}$ results in a connected graph that is isomorphic to the rSPR graph of $\cN$.
\end{proof}

Although it is well-known that each set $\cP$ of phylogenetic trees can be displayed by some phylogenetic network $\cN$ such that $\cP\subseteq T(\cN)$, it is not always possible to find a phylogenetic network with display set $\cP$. In this case, each phylogenetic network that displays $\cP$ infers additional phylogenetic trees that are not contained in $\cP$. It is therefore of interest to characterise sets of phylogenetic trees that are equal to the display set of some phylogenetic network. The next corollary, which follows from the contrapositive of Lemma ~\ref{l:1-connected}, makes a first step in this direction and gives a sufficient condition for when there exists no phylogenetic network whose display set is equal to a given set of phylogenetic trees.

\begin{corollary}\label{c:no-name}
Let $\cP$ be a set of phylogenetic $X$-trees. If the rSPR graph for $\cP$ is not connected, then there exists no phylogenetic network $\cN$ on $X$ such that $T(\cN)=\cP$.
\end{corollary}

\noindent In addition to the last corollary that does not impose any restrictions on the phylogenetic networks under consideration, Section~\ref{sec:level-1} establishes necessary and sufficient conditions for when a set of phylogenetic trees is the display set of a level-$1$ network.

Now, let $\cT$ and $\cT'$ be two phylogenetic $X$-trees. Accounting for $\rho$ in the definition of an agreement forest for $\cT$ and $\cT'$, there are $2|X|-1$ forests of size two that can be obtained from $\cT$ by deleting a single arc. Hence, it can be checked in polynomial time if $d_\rSPR(\cT,\cT')=1$ because, in this case, one of the $2|X|-1$  forests is guaranteed to be a maximum agreement forest for $\cT$ and $\cT'$. In turn, for an arbitrary-sized set $\cP$ of phylogenetic trees, it can be checked in time that is polynomial in $|\cP|$ and $|X|$ if the rSPR graph of $\cP$ is connected. Note that the converse of Corollary~\ref{c:no-name} is not true. For example, it is straightforward to check that each phylogenetic network that displays the three phylogenetic trees $\cT_{01}$, $\cT_{10}$, and $\cT_{00}$ that are shown in Figure~\ref{fig:rSPR-graph} has at least two reticulations and displays strictly more than three phylogenetic trees.

\section{Phylogenetic Networks with a Maximum Display Set} \label{sec:maximum}

In this section, we consider phylogenetic networks $\cN$ that have the property $|T(\cN)|=2^k$, where $k$ is the number of reticulations of $\cN$. As noted earlier the well-studied class of normal networks has this property~\cite{iersel10,willson12}. Moreover, as we will make more precise in the next section, if we ignore the trivial reticulations of a level-$1$ network $\cN$, then $\cN$ also has this property.

\begin{theorem}\label{t:hamiltonian-cycle}
Let $\cN$ be a phylogenetic network on $X$ with $k$ reticulations, and let $G$ be the rSPR graph of $\cN$. If $|T(\cN)|=2^k$, then the $k$-dimensional hypercube $Q_k$ is a spanning subgraph of $G$. In particular, $G$ has a \Blue{Hamilton} cycle if $k\ge 2$. 
\end{theorem}

\begin{proof}
Let  $(\cN,\phi,(v_1,v_2,\ldots,v_k))$ of $\cN$ be an ordered binary assignment of $\cN$, and let $n=2^k$. Let $B$ be the set of all $k$-bit strings, and let $(s_1,s_2,\ldots,s_{n})$ be an ordering on the elements in $B$ that is a Gray code. Now consider $Q_k$. By Theorem~\ref{t:code-iff-cube}, $\{\{s_1,s_2\},\{s_2,s_3\},\ldots,\{s_{n-1},s_{n}\},\{s_{n},s_1\}\}$ is the edge set of a \Blue{Hamilton} cycle in $Q_k$. 

We complete the proof by showing that $Q_k$ is a spanning subgraph of $G$. Let $$\psi:\{s_1,s_2,\ldots, s_n\}\rightarrow T(\cN)$$ be a map that assigns each $s_j$ with $j\in\{1,2,\ldots,n\}$, to the phylogenetic tree that is encoded by $s_j$ under $\phi$. As $T(\cN)$ is maximum, $\psi$ is a bijection. Consider an edge $\{s_j,s_{j'}\}$ in $Q_k$. As $d(s_j,s_{j'})=1$, it follows from Lemma~\ref{l:hamming1} that $d_\rSPR(\psi(s_j),\psi(s_{j'}))\leq 1$. In fact, again as $T(\cN)$ is maximum, we have $$d_\rSPR(\psi(s_j),\psi(s_{j'}))= 1.$$ Thus, if $\{s_j,s_{j'}\}$ is an edge in $Q_k$, then $\{\psi(s_j),\psi(s_{j'})\}$ is an edge in $G$. It now follows that, as $Q_k$ has a \Blue{Hamilton} cycle, so does $G$, thereby establishing the theorem.
\end{proof}

\noindent Referring back to Figure~\ref{fig:rSPR-graph}, observe that each of the two phylogenetic networks $\cN$ and $\cN'$ that is shown in this figure is normal and  has two reticulations. Furthermore the rSPR graph $G$ of $\cN$ is $Q_2$, and the rSPR graph of $\cN'$ is the complete graph on four vertices, which can be obtained from $Q_2$ by adding two additional edges.

\section{Characterising Display Sets of Level-$1$ Networks}\label{sec:level-1}

In this section, we characterise when a set $\cP$ of phylogenetic trees is the display set of a level-$1$ network. This characterisation is phrased in terms of the rSPR graph of $\cP$. Unlike the previous sections that gave explicit binary assignments as well as a mapping from the set of all bit strings of a given length to a collection of phylogenetic trees, for ease of reading, bijections between vertices of a hypercube and vertices of an rSPR graph are implicit in this and the next section. We start by giving some further definitions.

Let $\cT$ and $\cT'$ be two phylogenetic $X$-trees each with root $\rho$, and suppose that $d_\rSPR(\cT, \cT')=1$. We refer to a subset $X'$ of $X$ as a {\it moving subtree} for $\cT$ and $\cT'$ if the bipartition $$\{(X\cup\{\rho\})-X',X'\}$$ of $X\cup\{\rho\}$ is a maximum agreement forest for $\cT$ and $\cT'$. Intuitively, if $X'$ is a moving subtree for $\cT$ and $\cT'$, then $\cT'$ can be obtained from $\cT$ by pruning and regrafting the pendant subtree $\cT|X'$. Note that $\cT$ and $\cT'$ may not have a unique moving subtree. If $X'$ is a moving subtree for $\cT$ and $\cT'$, we associate this move with the ordered pair $(X', Y')$, where $Y'$ is the minimal cluster in $C(\cT)\cap C(\cT')$ properly containing $X'$.

Let $\cP$ be a set of phylogenetic $X$-trees such that $|\cP|=2^k$ for some non-negative integer $k$, and let $G$ be the rSPR graph of $\cP$. Suppose that there is a (graph) isomorphism from $G$ to $Q_k$. Under this isomorphism, for all $i\in \{1, 2, \ldots, k\}$, let $E_i$ denote the subset of edges of $G$ corresponding to the $i$-th bit edge subset of $Q_k$ throughout the remainder of the paper. Now label each edge $e$ of $G$ with the ordered pair that is associated with a moving subtree for the end vertices of $e$, and suppose that, for all $i\in \{1, 2, \ldots, k\}$, this labelling can be done so that each edge in $E_i$ has the same label, $(X_i, Y_i)$ say. Then $G$ is said to have the {\em nested subtree property} if, for all distinct $i, j\in \{1, 2, \ldots, k\}$, the ordered pairs $(X_i, Y_i)$ and $(X_j, Y_j)$ satisfy one of the following:
\begin{enumerate}[(I)]
\item $Y_i\cap Y_j=\emptyset$;

\item $Y_i\subseteq X_j$; and

\item $Y_i\subset Y_j$ and $X_j\cap Y_i=\emptyset$.
\end{enumerate}

\begin{figure}[t]
\center
\scalebox{1}{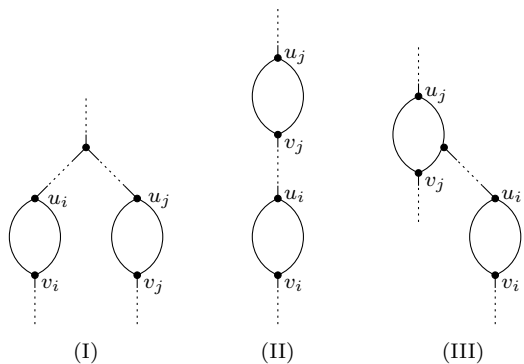}
\caption{The three ways in which the clusters of two reticulations $v_i$ and $v_j$ and the  clusters of their two respective source vertices $u_i$ and $u_j$ of a level-$1$ network can interact. Note that the ordered pairs $(C(v_i), C(u_i))=(X_i,Y_i)$ and $(C(v_j), C(u_j))=(X_j, Y_j)$ satisfy (I) in the definition of the nested subtree property for the level-$1$ network on the left-hand side, and (II) and (III) of the same definition for the level-$1$ network in the middle and right-hand side, respectively. Each solid arc of an underlying cycle indicates a directed path of arbitrary length.}
\label{fig:nested-subtree-prop}
\end{figure}

It is easily checked that, for all distinct ordered pairs $(X_i, Y_i)$ and $(X_j, Y_j)$, at most one of (I)--(III) holds, and if $(X_i, Y_i)$ and $(X_j, Y_j)$ are distinct and satisfy one of (I)--(III), then $Y_i\neq Y_j$. In this section, we will always view $G$ as having each of its edges labelled with the ordered pair associated with a corresponding moving subtree represented by the edge.

Properties (I)--(III) of the nested subtree property capture the way certain clusters of a level-$1$ network interact. In particular, let $\cN$ be a level-$1$ network, and let $u$ be a vertex of an underlying cycle $\cC$ of $\cN$, and let $v$ be the (unique) reticulation of $\cC$. If $u$ is the root of $\cN$ or no arc of $\cN$ that is directed into $u$ lies on $\cC$, then $u$ is called the {\it source vertex} of $v$. Since no two underlying cycles of $\cN$ intersect, it is easily seen that this notion is well defined. Now, if $v_i$ and $v_j$ are distinct reticulations of $\cN$, and $u_i$ and $u_j$ are the source vertices of $v_i$ and $v_j$, respectively, then it turns out that the ordered pairs $(C(v_i), C(u_i))$ and $(C(v_j), C(u_j))$ satisfy one of (I)--(III) \Blue{as illustrated in  Figure~\ref{fig:nested-subtree-prop}.}

We are now in a position to state the main result of this section. 

\begin{theorem}\label{t:main}
Let $\cP$ be a set of phylogenetic $X$-trees. Then $\cP$ is the display set of a level-$1$ network on $X$ if and only if, for some non-negative integer $k$, the rSPR graph of $\cP$ is isomorphic to $Q_k$ and has the nested subtree property.
\end{theorem}

Note that, in the statement of Theorem~\ref{t:main}, if $\cP$ is the display set of a level-$1$ network, then $|\cP|=2^k$ for some non-negative integer $k$.

\begin{figure}[t]
\center
\scalebox{1}{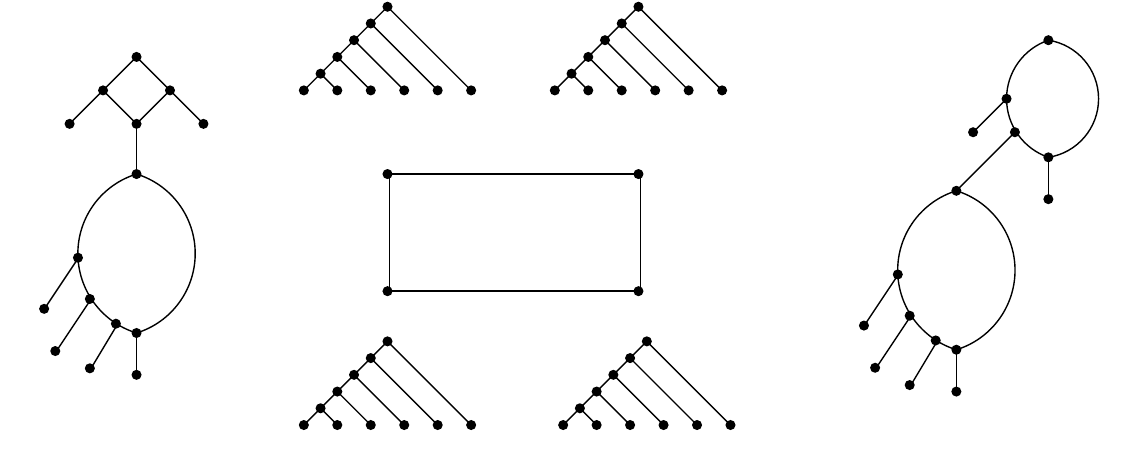}
\caption{The rSPR graph $G$ of a level-$1$ network $\cN$ with $T(\cN)=\{\cT_1,\cT_2, \cT_3, \cT_4\}$, and a level-$1$ network $\cN'$ such that $T(\cN)=T(\cN')$.}
\label{fig:level-1}
\end{figure}

\noindent {\bf Example.} Let $X=\{1, 2, 3, 4, 5, 6\}$, and consider the level-$1$ network $\cN$ on $X$ shown in Figure~\ref{fig:level-1}. The set $T(\cN)$ consists of the four phylogenetic trees $\cT_1$, $\cT_2$, $\cT_3$, and $\cT_4$ which are illustrated as the vertices of the rSPR graph $G$ of this set in the same figure. Now, $\{1, 2, 3, 4\}$ is a moving subtree for $\cT_1$ and $\cT_2$ as well as for $\cT_3$ and $\cT_4$. In both instances, the ordered pair corresponding to this moving subtree is 
$$(\{1, 2, 3, 4\}, X).$$
Similarly, $\{1\}$ is a moving subtree for $\cT_1$ and $\cT_3$ and for $\cT_2$ and $\cT_4$. The corresponding ordered pair in both instances is
$$(\{1\}, \{1, 2, 3, 4\}).$$
Since $(\{1, 2, 3, 4\}, X)$ and $(\{1\}, \{1, 2, 3, 4\})$ satisfy (II), and there are only two ordered pairs to compare, it follows that $G$ has the nested subtree property. Note that $(C(v_1), C(u_1))=(\{1, 2, 3, 4\}, X)$ and $(C(v_2), C(u_2))=(\{1\}, \{1, 2, 3, 4\})$.

Now observe that we could instead have labelled the edges $\{\cT_1, \cT_2\}$ and $\{\cT_3, \cT_4\}$ of $G$ with the ordered pair $(\{5\}, X)$, in which case, $(\{5\}, X)$ and $(\{1\}, \{1, 2, 3, 4\})$ satisfy (III). Thus the choice of ordered pair for the edges $\{\cT_1, \cT_2\}$ and $\{\cT_3, \cT_4\}$ is not unique. Indeed, in the proof of Theorem~\ref{t:main} this choice leads to the construction of another level-$1$ network $\cN'$ whose display set is also $\{\cT_1, \cT_2, \cT_3, \cT_4\}$. By way of comparison, $\cN'$ is shown in Figure~\ref{fig:level-1}.

The remainder of this section establishes Theorem~\ref{t:main}. We first provide some additional terminology and preliminary results. Subsequently, we establish separately the two directions of Theorem~\ref{t:main} as Lemmas~\ref{main1} and~\ref{main2}.

Let $v$ be a reticulation of a level-$1$ network $\cN$. We say that $v$ is {\it non-trivial} if the unique underlying cycle of $\cN$ that contains $v$ has at least four vertices and, otherwise, we say that $v$ is {\it trivial}. Let $v$ be a trivial reticulation of $\cN$. Obtain a phylogenetic network $\cN'$ from $\cN$ by deleting one of the two arcs directed into $v$ and suppressing the two resulting degree-2 vertices.  (If one of the two arcs is incident with the root of $\cN$, choose the other arc to delete.) As $\cN$ is level-$1$, so is $\cN'$. Repeating this step for each remaining trivial reticulation in $\cN'$ results in a level-$1$ network, say $\cN^*$, with no trivial reticulation. We refer to $\cN^*$ as the {\it essential level-$1$ network} with respect to $\cN$. Since no two underlying cycles of $\cN$ have a common vertex, $\cN^*$ is unique. Moreover, we have the following observation. 

\begin{observation}\label{ob:essential}\cite[Theorem 3.1]{linz}
Let $\cN$ be a level-$1$ network and let $\cN^*$ be the essential level-$1$ network with respect to $\cN$. Then $T(\cN)=T(\cN^*)$. Moreover, if $\cN^*$ has $k$ reticulations, then $|T(\cN^*)|=2^k$.
\end{observation}

The next corollary is an immediate consequence of Theorem~\ref{t:hamiltonian-cycle} and Observation~\ref{ob:essential} as well as an immediate consequence of Theorem~\ref{t:main}.

\begin{corollary}\label{c:hamiltonian-cycle}
Let $G$ be the rSPR graph of a level-$1$ network with at least two non-trivial reticulations. Then $G$ has a \Blue{Hamilton} cycle.
\end{corollary}

\begin{figure}[t]
\center
\scalebox{1}{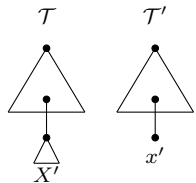}
\caption{A generic example of the subtree reduction that reduces a phylogenetic $X$-tree $\cT$ to a phylogenetic tree $\cT'$ on leaf set $(X-X')\cup\{x'\}$. Triangles indicate subtrees.}
\label{fig:subtree-reduction}
\end{figure}

Let $\cT$ be a phylogenetic $X$-tree, and suppose that $X'$ is a subset of $X$ such that $X'$ is a cluster of $\cT$. Let $\cT'$ be the phylogenetic tree obtained from $\cT$ by replacing the pendant subtree whose leaf set is $X'$ with a new leaf, $x'$ say. That is, $\cT'$ is obtained from $\cT$ by deleting all vertices $v$ (and their incident arcs) such that $C(v)$ is a proper subset of $X'$ and label the vertex $u$ whose cluster is $X'$ with $x'$. Note that the leaf set of $\cT'$ is $(X-X')\cup \{x'\}$. We say that $\cT'$ has been obtained from $\cT$ by a {\em subtree reduction on $X'$}. The leaf $x'$ is referred to as the {\em replacement leaf}. \Blue{A generic example of a subtree reduction is shown in Figure~\ref{fig:subtree-reduction}.}

\begin{lemma}
\cite[Proposition~3.2]{bordewich05} Let $\cT$ and $\cT'$ be two distinct phylogenetic $X$-trees, and suppose that $Y\subseteq X$ such that $\cT|Y\cong \cT'|Y$. Let $\cT_1$ and $\cT'_1$ be the phylogenetic trees obtained from $\cT$ and $\cT'$, respectively, by applying a subtree reduction on $Y$ with new replacement leaf $y$. Then $d_{\rm rSPR}(\cT, \cT')=1$ if and only if $d_{\rm rSPR}(\cT_1, \cT'_1)=1$. In particular, $\{(X\cup \{\rho\})-X', X'\}$ is an agreement forest for $\cT$ and $\cT'$ if and only if either
\begin{enumerate}[{\rm (i)}]
\item $Y\subseteq X'$ and $\{(X\cup \{\rho\})-X', (X'-Y)\cup \{y\}\}$ is an agreement forest for $\cT_1$ and $\cT'_1$, or

\item $Y\subseteq (X\cup \{\rho\})-X'$ and $\{((X\cup \{\rho\})-(X'\cup Y))\cup \{y\}, X'\}$ is an agreement forest for $\cT_1$ and $\cT'_1$.
\end{enumerate}
\label{subtree}
\end{lemma}

We are now ready to prove the two directions of Theorem~\ref{t:main} beginning with the necessary direction.

\begin{lemma}
Let $\cN$ be a level-$1$ network on $X$, and let $k$ be the number of non-trivial reticulations of $\cN$. Then the rSPR graph of $\cN$ is isomorphic to $Q_k$ and has the nested subtree property.
\label{main1}
\end{lemma}

\begin{proof}
By Observation~\ref{ob:essential} and the paragraph prior to it, we may assume that $\cN$ has no trivial reticulations. Let $\{v_1, v_2, \ldots, v_k\}$ denote the set of reticulations of $\cN$ and, for all $i\in \{1, 2, \ldots, k\}$, let $u_i$ denote the source vertex of $v_i$. Furthermore, let $X_i$ and $Y_i$ denote the clusters $C(v_i)$ and $C(u_i)$, respectively, of $\cN$. For the proof of the lemma, we will prove a stronger statement. In particular, we will additionally show that there is an isomorphism that maps the rSPR graph of $\cN$ to $Q_k$ such that, for all $i\in \{1, 2, \ldots, k\}$, the subset of edges of $G$ corresponding to the $i$-th bit edge subset of $Q_k$ can each be labelled $(X_i, Y_i)$, and that this choice of labelling verifies that $G$ has the nested subtree property. The proof is by induction on $k$. If $k=0$, then $\cN$ is a phylogenetic $X$-tree and the rSPR graph of $\cN$ is isomorphic to $Q_0$, and the stronger statement, and thus the lemma, immediately follows. If $k=1$, then the rSPR graph of $\cN$ is isomorphic to $Q_1$, and the stronger statement, and therefore the lemma, immediately follows again. Now suppose that $k\ge 2$ and the stronger statement holds for all level-$1$ networks on $X$ with at most $k-1$ reticulations.

\Blue{For each $i\in \{1, 2, \ldots, k\}$, recall that $u_i$ is the source  vertex of $v_i$.} Without loss of generality, we may assume that $u_k$ is a source vertex at maximum distance from the root of $\cN$.  Let $p_1$ and $p_2$ denote the parents of $v_k$. Since $k\ge 2$, it follows by the maximality of $u_k$ that neither $p_1$ nor $p_2$ is the root of $\cN$. If either $(p_1, v_k)$ or $(p_2, v_k)$ is a shortcut, we may assume that $(p_1, v_k)$ is the shortcut. Note that at most one of $(p_1, v_k)$ and $(p_2, v_k)$ is a shortcut; otherwise, $(p_1, v_k)$ and $(p_2, v_k)$ are parallel arcs. Let $\cN_1$ be the level-$1$ network on $X$ obtained from $\cN$ by deleting the arc $(p_2, v_k)$ and suppressing the two resulting degree-two vertices. Since $\cN$ has $k$ reticulations, $\cN_1$ has $k-1$ reticulations and it follows by the induction assumption that there is an isomorphism $\varphi_1$ that maps the rSPR graph $G_1$ of $\cN_1$ to $Q_{k-1}$ such that, for all $i\in \{1, 2, \ldots, k-1\}$, if $E^1_i$ denotes the subset of edges of $G_1$ corresponding to the $i$-th bit edge subset of $Q_{k-1}$, then we can label each edge in $E^1_i$ with $(X_i, Y_i)$ and this labelling verifies that $G_1$ has the required nested subtree property of the stronger statement. Since $\cN$ is level-$1$, for each $i\in \{1, 2, \ldots, k-1\}$, the clusters of $u_i$ and $v_i$ are $Y_i$ and $X_i$, respectively, in $\cN_1$. \Blue{This setup is illustrated in Figure~\ref{fig:G_1} for when $k=3$, and neither $(p_1,v_3)$ nor $(p_2,v_3)$ is a shortcut in $\cN$. Furthermore, the same figure illustrates the rest of the inductive proof.}

\begin{figure}[t]
\center
\scalebox{1}{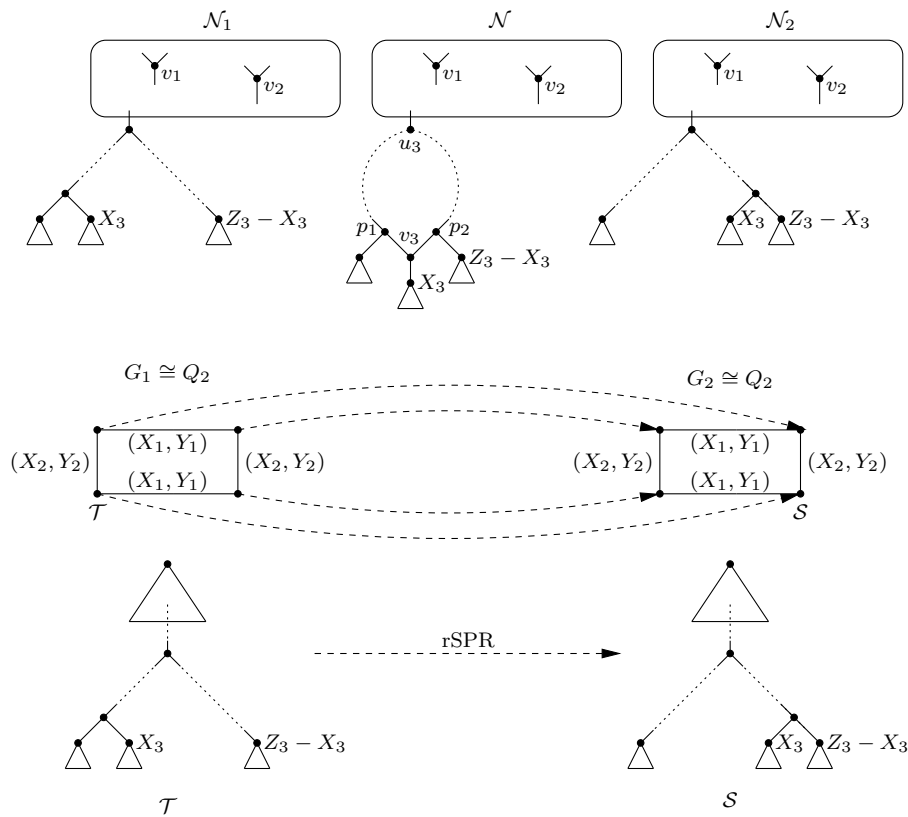}
\caption{Setup as described in the proof of Lemma~\ref{main1} for when $k=3$, and neither $(p_1,v_3)$ nor $(p_2,v_3)$ is a shortcut in $\cN$. Triangles indicate subtrees of their respective level-$1$ network or phylogenetic tree. Furthermore, vertices whose clusters are $X_3$ or $Z_3-X_3$ are labelled accordingly.}
\label{fig:G_1}
\end{figure}

Now let $\cN_2$ be the level-$1$ network on $X$ obtained from $\cN$ by deleting the arc $(p_1, v_k)$ and suppressing the two resulting degree-two vertices. We next construct from $G_1$ the rSPR graph of $\cN_2$. By the maximality of $u_k$, if $\cT$ and $\cT'$ are vertices of $G_1$, then $\cT|Y_k\cong \cT'|Y_k$. Let $Z_k$ denote the cluster $C(p_2)$ of $\cN$. Note that, as $(p_2, v_k)$ is not a shortcut, $X_k\subset Z_k\subseteq Y_k$. Let $G_2$ be the graph obtained from $G_1$ by replacing each vertex $\cT$ of $G_1$ with the phylogenetic $X$-tree $\cS$ obtained from $\cT$ by a single rSPR operation that prunes the pendant subtree whose leaf set is $X_k$ and regrafts it to the arc directed into the vertex whose cluster is $Z_k-X_k$. \Blue{Again, this is illustrated in Figure~\ref{fig:G_1}.} For ease of reading, we say that $\cS$ is the vertex in $G_2$ whose {\em partner} is $\cT$. Evidently, $G_2$ is isomorphic to $Q_{k-1}$ under the isomorphism $\varphi_2$ that is obtained from $\varphi_1$ by replacing every vertex $\cT$ in $G_1$ with its partner $\cS$ in $G_2$. Furthermore, by construction, if $\cS$ and $\cS'$ are vertices of $G_2$, then $\cS|Y_k\cong \cS'|Y_k$.

We next show that $G_2$ is the rSPR graph of $\cN_2$ and the labelling of its edges verifies the stronger statement. Since the vertex set of $G_1$ is the display set of $\cN_1$, it follows by the maximality of $u_k$ and construction that the vertex set of $G_2$ is the display set of $\cN_2$. Let $\cT$ and $\cT'$ be vertices of $G_1$, and let $\cS$ and $\cS'$ be the partners of $\cT$ and $\cT'$, respectively, in $G_2$. Let $\cT_1$, $\cT'_1$, $\cS_1$, and $\cS'_1$ denote the phylogenetic trees obtained from $\cT$, $\cT'$, $\cS$, and $\cS'$, respectively, by applying a subtree reduction on $Y_k$ with replacement leaf $y_k$. First assume that $\cT$ and $\cT'$ are adjacent in $G_1$, and let $(X', Y')$ denote the ordered pair labelling the edge joining $\cT$ and $\cT'$. Then $\{(X\cup \{\rho\})-X', X'\}$ is an agreement forest for $\cT$ and $\cT'$, and $Y'$ is the minimal cluster in $C(\cT)\cap C(\cT')$ properly containing $X'$. Since $\cT|Y_k\cong \cT'|Y_k$, it follows by Lemma~\ref{subtree} that either (i) $Y_k\subseteq X'$ and
$$\{(X\cup \{\rho\})-X', (X'-Y_k)\cup \{y_k\}\}$$
is an agreement forest for $\cT_1$ and $\cT'_1$, or (ii) $Y_k\subseteq (X\cup \{\rho\})-X'$ and
$$\{((X\cup \{\rho\})-(X'\cup Y_k))\cup \{y_k\}, X'\}$$
is an agreement forest for $\cT_1$ and $\cT'_1$. If (i) holds, then $(Y'-Y_k)\cup \{y_k\}$ is the minimal cluster in $C(\cT_1)\cap C(\cT'_1)$ properly containing $(X'-Y_k)\cup \{y_k\}$. Furthermore, if (ii) holds, then, by the maximality of $u_k$, either $Y_k\subset Y'$, in which case, $(Y'-Y_k)\cup \{y_k\}$ is the minimal cluster in $C(\cT_1)\cap C(\cT'_1)$ properly containing $X'$, or $Y_k\cap Y'=\emptyset$, in which case, $Y'$ is the minimal cluster in $C(\cT_1)\cap C(\cT'_1)$ properly containing $X'$.

Now, by the single rSPR operation in which $\cS$ is obtained from $\cT$ and $\cS'$ is obtained from $\cT'$, it follows that $\cT_1\cong \cS_1$ and $\cT'_1\cong \cS'_1$. Thus, as $\cS|Y_k\cong \cS'|Y_k$, it follows by Lemma~\ref{subtree} that
$$\Green{\{(X\cup \{\rho\})-X', X'\}}$$
is an agreement forest for $\cS$ and $\cS'$, and $Y'$ is the minimal cluster in $C(\cS)\cap C(\cS')$ properly containing $X'$. Hence $\cS$ and $\cS'$ are correctly joined by an edge labelled $(X', Y')$ in $G_2$. Moreover, if $\cT$ and $\cT'$ are not adjacent in $G_1$, then $d_{\rm rSPR}(\cT, \cT') > 1$ and so, by Lemma~\ref{subtree}, $d_{\rm rSPR}(\cT_1, \cT'_1) > 1$. Thus, as $\cT_1\cong \cS_1$ and $\cT'_1\cong \cS'_1$, we have $d_{\rm rSPR}(\cS_1, \cS'_1) > 1$. Therefore, by Lemma~\ref{subtree}, $d_{\rm rSPR}(\cS, \cS') > 1$. We deduce that $G_2$ is the rSPR graph of $\cN_2$ and, as the labelling of the edges of $G_1$ verifies that $G_1$ has the nested subtree property, the labelling of the edges of $G_2$ has the required nested subtree property of the stronger statement.

We now construct the rSPR graph of $\cN$ from $G_1$ and $G_2$ as follows. Take $G_1$ and $G_2$ and, for each vertex $\cT$ in $G_1$, join $\cT$ to its partner $\cS$ in $G_2$ and label the edge $(X_k, Y_k)$. Call the resulting graph $G$. By construction, there is an isomorphism that maps $G$ to $Q_k$ such that, for all $i\in \{1, 2, \ldots, k\}$ each edge in the subset of edges of $G$ corresponding to the $i$-th bit subset of $Q_k$ is labelled $(X_i, Y_i)$. Furthermore, the end vertices of an edge joining a vertex $\cT$ in $G_1$ with a vertex $\cS$ in $G_2$ have an agreement forest $\{(X\cup \{\rho\})-X_k, X_k\}$, that is, $d_{\rm rSPR}(\cT, \cS)=1$, and $Y_k$ is the minimal cluster in $C(\cT)\cap C(\cS)$ containing $X'$.

We next show that if $\cT$ is a vertex in $G_1$ and $\cS$ is a vertex in $G_2$, but $\cS$ is not the partner of $\cT$, then $d_{\rm rSPR}(\cT, \cS) > 1$. Say $\cT$ and $\cS$ are such vertices, but $d_{\rm rSPR}(\cT, \cS)=1$. Then $\cT$ and $\cS$ have an agreement forest $\{(X\cup \Green{\{\rho\}})-Z, Z\}$, where $Z$ is a cluster of $\cT$ and $\cS$, and $\cT|Z\cong \cS|Z$. Since $\cT$ and $\cS$ have pendant subtrees with leaf set $Y_k$, but $\cT|Y_k\not\cong \cS|Y_k$, it follows that $Z\subset Y_k$. As $\cT$ is a vertex of $G_1$ and $\cS$ is a vertex of $G_2$, this implies that $Z=X_k$, in which case, $\cS$ is the partner of $\cT$, a contradiction. Thus, as $G_1$ and $G_2$ are the rSPR graphs of $\cN_1$ and $\cN_2$, respectively, \Blue{and a phylogenetic tree is displayed by $\cN$ if and only if it is either displayed by $\cN_1$ or displayed by $\cN_2$, it follows that} $G$ is the rSPR graph of $\cN$.

Lastly, let $i$ and $j$ be distinct elements of $\{1, 2, \ldots, k\}$. If $k\not\in \{i, j\}$, then, by construction, $(X_i, Y_i)$ and $(X_j, Y_j)$ satisfy one of (I)--(III) in the definition of the nested subtree property. So assume that $i=k$. Now $u_j$ is the source vertex of $v_j$, and $C(u_j)=Y_j$ and $C(v_j)=X_j$ in $\cN$. Say $Y_j\cap Y_k\ne\emptyset$. Then, by the choice of $u_k$, we have that $Y_k\subset Y_j$. Since $\cN$ is level-$1$, it is easily seen that either $Y_k\subseteq X_j$ or $X_j\cap Y_k=\emptyset$, and so $(X_j, Y_j)$ and $(X_k, Y_k)$ satisfy either (II) or (III). \Blue{Thus,} $G$ is the rSPR graph of $\cN$ and the labelling of the edges of $G$ has the required nested subtree property of the stronger statement. This completes the proof of the lemma.
\end{proof}

The next lemma is the converse of Lemma~\ref{main1}, and thereby completes the proof of Theorem~\ref{t:main}.

\begin{lemma}
Let $\cP$ be a set of phylogenetic $X$-trees and let $G$ be the rSPR graph of $\cP$. If $G$ is isomorphic to $Q_k$ for some non-negative integer $k$, and $G$ has the nested subtree property, then there is a level-$1$ network $\cN$ on $X$ whose display set is $\cP$.
\label{main2}
\end{lemma}

\begin{proof}
Suppose that there is an isomorphism from $G$ to $Q_k$ and that $G$ has the nested subtree property. Under this isomorphism, we may assume that, for all $i\in \{1, 2, \ldots, k\}$, each edge in $E_i$, the subset of edges of $G$ corresponding to the $i$-th bit edge subset of $Q_k$, has the label $(X_i, Y_i)$ and, for all distinct $i, j\in \{1, 2, \ldots, k\}$ the ordered pairs $(X_i, Y_i)$ and $(X_j, Y_j)$ satisfy one of (I)--(III). Note that, as $G$ is isomorphic to $Q_k$, we have $|\cP|=2^k$ for some non-negative integer $k$. The proof is by induction on $k$. If $k=0$, then $|\cP|=1$, and the lemma trivially holds by choosing $\cN$ to be the phylogenetic $X$-tree in $\cP$. Now suppose that $k\ge 1$ and \Blue{that} the lemma holds for all sets of phylogenetic trees of size $2^{k-1}$ with the same leaf set.

Amongst the ordered pairs $(X_1, Y_1), (X_2, Y_2), \ldots, (X_k, Y_k)$, choose $(X_i, Y_i)$ to be an ordered pair with the property that there is no $(X_j, Y_j)$ \Blue{for} which $Y_j\subset Y_i$. Such an ordered pair exists, otherwise $Y_i=Y_j$ for  some $i\neq j$ contradicting the assumption that the ordered pairs satisfy the nested subtree property. Furthermore, if there is an ordered pair $(X_j, Y_j)$ such that $X_j\cap Y_i\neq \emptyset$, then, as $G$ satisfies the nested subtree property, $Y_i\subseteq X_j$.

Consider the graph obtained from $G$ by deleting the edges in $E_i$. The resulting graph has exactly two components, $G_1$ and $G_2$ say, and each component is isomorphic to $Q_{k-1}$. Observe that, as $G$ has the nested subtree property, $G_1$ and $G_2$ also have this property. We next show that if $\cT$ and $\cT'$ are vertices of $G_1$, then $\cT|Y_i\cong \cT'|Y_i$. Certainly, $Y_i\in C(\cT)$ and $Y_i\in C(\cT')$ as all vertices of $G_1$ are incident with an edge labelled $(X_i, Y_i)$ in $G$. Assume that $\cT$ and $\cT'$ are adjacent in $G_1$. Let $(X_j, Y_j)$ denote the ordered pair labelling the edge joining $\cT$ and $\cT'$ in $G_1$, where $i\neq j$. Then
$$\{(X\cup \{\rho\})-X_j, X_j\}$$
is an agreement forest for $\cT$ and $\cT'$.  By the choice of $(X_i, Y_i)$, one of (I) $Y_i\cap Y_j=\emptyset$, (II) $Y_i\subseteq X_j$, and (III) $Y_i\subset Y_j$ and $X_j\cap Y_i=\emptyset$ holds. Thus either $Y_i\subseteq (X\cup \{\rho\})-X_j$ or $Y_i\subseteq X_j$. Therefore, as $\{(X\cup \{\rho\})-X_j, X_j\}$ is an agreement forest for $\cT$ and $\cT'$, it follows that in all cases $\cT|Y_i\cong \cT'|Y_i$. Since $Q_{k-1}$ is connected, we can repeatedly apply this argument to eventually show that $\cT|Y_i$ and $\cT'|Y_i$ are isomorphic for all vertices $\cT$ and $\cT'$ in $G_1$. Similarly, if $\cS$ and $\cS'$ are vertices of $G_2$, then $\cS|Y_i\cong \cS'|Y_i$. \Blue{For $k=i=3$ and $j=1$, the preceding argument is illustrated in Figure~\ref{fig:agreement-forest}.}

\begin{figure}[t]
\center
\scalebox{1}{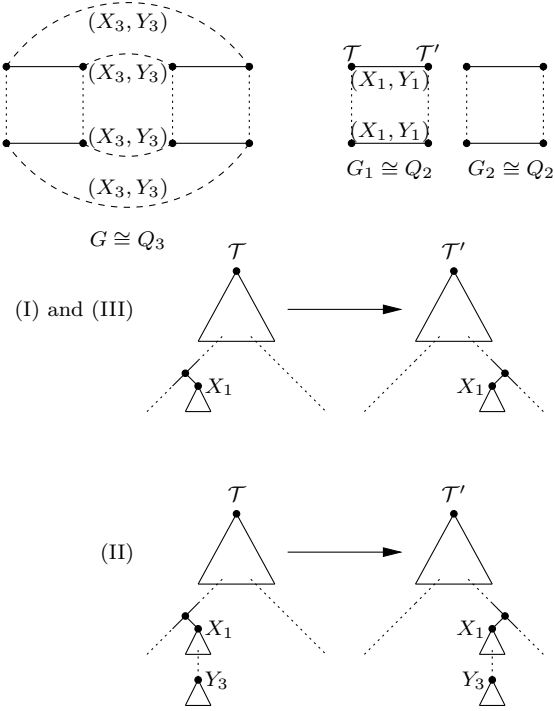}
\caption{Setup as described in the proof of Lemma~\ref{main2} for when $k=i=3$, $j=1$, and $E_1$ contains the solid edges of $G$, $E_2$ contains the dotted edges of $G$, and $E_3$ contains the dashed edges of $G$. Vertices whose clusters are $X_1$ or $Y_3$ are labelled accordingly. If (I) or (III) in the definition of the nested subtree property applies to $(X_1,Y_1)$ and $(X_3,Y_3)$, then $X_1\cap Y_3=\emptyset$. Otherwise, if (II) applies, then $Y_3\subseteq X_1$.
As $\cT$ and $\cT'$ are joined by an edge in $G_1$ that is labelled $(X_1,Y_1)$, it follows that $\cT'$ can be obtained from $\cT$ by a single rSPR \Green{operation} that prunes and regrafts the  pendant subtree $\cT|X_1$. Hence, regardless of which of (I)--(III) applies, as $\{(X\cup\{\rho\})-X_1, X_1\}$ is an agreement forest for $\cT$ and $\cT'$, where $X$ is the leaf set of $\cT$ and $\cT'$, we have $\cT|Y_3\cong \cT'|Y_3$.}
\label{fig:agreement-forest}
\end{figure}

Let $\cP_1$ denote the set of vertices of $G_1$, and let $\cP'_1$ denote the collection of phylogenetic trees on $(X-Y_i)\cup \{y_i\}$ obtained by replacing each phylogenetic $X$-tree in $\cP_1$ with the phylogenetic tree on $(X-Y_i)\cup \{y_i\}$ resulting from a subtree reduction on $Y_i$, where the replacement leaf is $y_i$. By Lemma~\ref{subtree}, the rSPR graph $G'_1$ of $\cP'_1$ can be obtained from $G_1$ by replacing each vertex with the phylogenetic tree in $\cP'_1$ resulting from this subtree reduction, and replacing those ordered pairs $(X_j, Y_j)$ in which $Y_i\subseteq X_j$ with
$$((X_j-Y_i)\cup \{y_i\}, (Y_j-Y_i)\cup \{y_i\})$$
and in which $Y_i\subset Y_j$ and $X_j\cap Y_i=\emptyset$ with
$$(X_j, (Y_j-Y_i)\cup \{y_i\}).$$
Using the fact that $G_1$ has the nested subtree property, a routine check shows that $G'_1$ also has the nested subtree property and so, by the induction assumption, there is a level-$1$ network $\cN'_1$ on $(X-Y_i)\cup \{y_i\}$ whose display set is $\cP'_1$. Now let $\cN_1$ be the level-$1$ network on $X$ obtained from $\cN'_1$ by replacing $y_i$ with the (pendant) subtree $\cT|Y_i$, where $\cT\in \cP_1$. That is, $\cN_1$ is obtained from $\cN'_1$ by identifying the root of a phylogenetic tree isomorphic to $\cT|Y_i$ with the vertex $y_i$. Since the display set of $\cN'_1$ is $\cP'_1$, it follows that the display set of $\cN_1$ is $\cP_1$.

Let $\cS$ be a vertex of $G_2$, and recall that if $\cS'$ is also a vertex of $G_2$, then $\cS|Y_i\cong \cS'|Y_i$. Let $\cT$ be the unique vertex of $G_1$ such that $\{\cT,\cS\}$ is an edge in $G$. Furthermore, let $u$ (resp. $u'$) be the vertex of $\cT$ (resp. $\cS$) such that $X_i\subset C(u)$ (resp.  $X_i\subset C(u')$) and $u$ (resp. $u'$) has no child whose cluster is a proper superset of $X_i$. Note that $C(u)\subseteq Y_i$ and $C(u')\subseteq Y_i$ as $(X_i, Y_i)$ labels the edge $\{\cT,\cS\}$. Let $\cN$ be the phylogenetic network obtained from $\cN_1$ in one of the following two \Blue{ways:}
\begin{enumerate}[(i)]
\item If $(C(u)-X_i)\cap (C(u')-X_i)=\emptyset$ or $C(u')\subseteq C(u)$, subdivide the arc directed into the (unique) vertex whose cluster is $C(u')-X_i$, subdivide the arc directed into the (unique) vertex whose cluster is $X_i$, and adjoin an arc from the first to the second of these subdivisions.
\item If $C(u)\subseteq C(u')$, subdivide the arc directed into the (unique) vertex whose cluster is $C(u')$\Blue{,} subdivide the arc directed into the (unique) vertex whose cluster is $X_i$, and adjoin an arc from the first to the second of these subdivisions.
\end{enumerate}
The arcs that get subdivided in (i) and (ii) exist as $X_i$ is a cluster of $\cT$ and $\cS$, and $\cT|((X\cup \{\rho\})-X_i)\cong \cS|((X\cup \{\rho\})-X_i)$. By the choice of $(X_i, Y_i)$, the network $\cN$ is level-$1$. If $\cT$ is the unique vertex of $G_1$ in $G$ adjacent to $\cS$, then $\cT|(X-X_i)\cong\cS|(X-X_i)$ and $\cT|X_i\cong \cS|X_i$. Therefore, by construction, $\cN$ displays $\cS$. Since the choice of $\cS$ is arbitrary in the sense that $\cS|Y_i\cong S'|Y_i$ for all vertices $\cS'$ of $G_2$, the lemma now follows.
\end{proof}

\section{Reconstructing Level-$1$ Networks for a Given Display Set}\label{sec:level-1-algo}

Using the characterisation established in the last section, this section describes a polynomial-time algorithm, called {\sc Construct Level-$1$ Network}, that, given a collection $\cP$ of phylogenetic trees, reconstructs a level-$1$ network whose display set is $\cP$ or returns that no such network exists. As a corollary, we derive that in fact all such networks can be reconstructed. Before we present the algorithm, we establish several results on the properties of rSPR graphs and hypercubes that are needed to show that {\sc Construct Level-$1$ Network} runs in polynomial time.

Let $\cT$ and $\cT'$ be two phylogenetic $X$-trees, and suppose that $\cT'$ can be obtained from $\cT$ by a single rSPR operation. In particular, $\cT'$ can be obtained from $\cT$ by deleting an arc $(u, v)$ and reattaching the resulting rooted subtree that contains $v$ with a new arc $(u', v)$. Ignoring the suppressing of $u$, if the underlying path joining $u$ and $u'$ consists of at most two arcs, then we say that $\cT'$ has been obtained from $\cT$ by a {\em rooted nearest neighbour interchange} (rNNI) operation. Note that if this path consists of one arc, then $\cT\cong \cT'$. The {\em rNNI distance} between any two phylogenetic $X$-trees $\cT$ and $\cT'$ is the minimum number of rNNI operations that transforms $\cT$ into $\cT'$, and is denoted by $d_{\rm rNNI}(\cT, \cT')$. Like the rSPR operation, rNNI is reversible and $d_{\rm rNNI}(\cT, \cT')$ is well defined as there is always a sequence of rNNI operations that transforms $\cT$ into $\cT'$~\cite{bordewich05,robinson71}. 

Again, let $\cT$ and $\cT'$ be two phylogenetic $X$-trees with $d_\rSPR(\cT,\cT')=1$. We denote the set of moving subtrees for $\cT$ and $\cT'$ by $M(\cT, \cT')$. Furthermore, for a cluster $X'$ of $\cT$, we call the vertex, $v$ say, of $\cT$, such that $C(v)=X'$, the {\em parent} of $X'$ and the vertex, $u$ say, of $\cT$, such that $(u, v)$ is the arc of $\cT$ directed into $v$, the {\em grandparent} of $X'$. Lastly, if $X'\subseteq X$ and $|X'|=3$, then $\cT|X'$ is called a {\em rooted triple} of $\cT$. If $X'=\{a, b, c\}$ and the underlying paths joining $a$ and $b$, and joining the root of $\cT$ and $c$ are disjoint, then we denote the rooted triple $\cT|\{a, b, c\}$ by $ab|c$ or, equivalently, $ba|c$. The set of rooted triples of $\cT$ is denoted by $\cR(\cT)$. It is well known that $\cR(\cT)=\cR(\cT')$ if and only if $\cT\cong \cT'$ (see, for example, \cite{semple03}).

\begin{figure}[t]
\center
\scalebox{1}{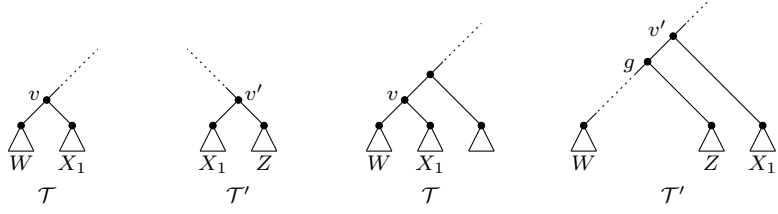}
\caption{Setup of $\cT$ and $\cT'$ as used in the proof of Proposition~\ref{count} depending on whether $X_1$ is a moving subtree of Type~I (left) or Type~II (right). Triangles indicate pendant subtrees with at least one leaf.}
\label{fig:type-1-2}
\end{figure}

The next proposition bounds the number of moving subtree for two phylogenetic trees. More specifically, it shows that, if two phylogenetic trees have rSPR distance one, then there exists a unique moving subtree unless the two trees also have rNNI distance one in which case there are exactly three moving subtrees.

\begin{proposition}\label{p:one-or-three}
Let $\cT$ and $\cT'$ be two phylogenetic $X$-trees, and suppose that $d_{\rm rSPR}(\cT, \cT')=1$. Then $|M(\cT, \cT')|\in \{1, 3\}$. Moreover, $|M(\cT, \cT')|=3$ if and only if $d_{\rm rNNI}(\cT, \cT')=1$.
\label{count}
\end{proposition}

\begin{proof}
Let $\rho$ denote the root of $\cT$ and $\cT'$. Since $d_{\rm rSPR}(\cT, \cT')=1$, it follows that $|M(\cT, \cT')|\ge 1$. Let $X_1\in M(\cT, \cT')$, and let $v$ and $v'$ be the grandparents of $X_1$ in $\cT$ and $\cT'$, respectively. Furthermore, let $u$ and $u'$ be the vertices of $\cT$ and $\cT'$, respectively, such that $C_{\cT}(u)$ and $C_{\cT'}(u')$ is the minimal cluster in $C(\cT)\cap C(\cT')$ properly containing $X_1$. Let $W=C_{\cT}(v)-X_1$. Since rNNI is reversible, we may view the rSPR operations corresponding to the moving subtrees in $M(\cT, \cT')$ as transforming $\cT$ into $\cT'$. Furthermore, we may also assume that either $W\cap C_{\cT'}(v')=\emptyset$, in which case, we say $X_1$ is a {\em Type~I} moving subtree, or $W\subseteq C_{\cT'}(v')$, in which case, we say $X_1$ is a {\em Type~II} moving subtree. 

Now, if $X_1$ is a Type~I moving subtree, let $Z=C_{\cT'}(v')-X_1$, and if $X_1$ is a Type~II moving subtree, let $Z$ be the cluster of $\cT'$ such that $Z\cap (X_1\cup W)=\emptyset$, and the grandparent, $g$ say, of $Z$ in $\cT'$ is on the path from $v'$ to the parent of $W$ and $(v', g)$ is an arc in $\cT'$. The setup of $\cT$ and $\cT'$ depending on whether $X_1$ is a moving subtree of Type~I or Type~II is illustrated in Figure~\ref{fig:type-1-2}. Note that $g$ exists; otherwise, $d_{\rm rSPR}(\cT, \cT')=0$. Observe that $W$ and $Z$ are both clusters of $\cT$ and $\cT'$, and $X_1$, $W$, and $Z$ are (pairwise) disjoint subsets of $X$. Moreover, if $x_1\in X_1$, $w\in W$, and $z\in Z$, then $x_1w|z\in \cR(\cT)$, and either $x_1z|w\in \cR(\cT')$ if $X_1$ is a Type I moving subtree or $wz|x_1\in \cR(\cT')$ if $X_1$ is a Type II moving subtree.

Now suppose that $X_2\in M(\cT, \cT')-\{X_1\}$. It follows by the observations at the end of the last paragraph that $(X_1\cup W\cup Z)\cap X_2\neq \emptyset$. Otherwise, $\{(X\cup \{\rho\})-X_2, X_2\}$ is not an agreement forest for $\cT$ and $\cT'$ as $x_1w|z$ is a rooted triple of $\cT|((X\cup \{\rho\})-X_2)$ but $x_1w|z$ is not a rooted triple of $\cT'|((X\cup \{\rho\})- X_2)$. Furthermore, using the same observations, a similar analysis establishes that either $X_2=W$, or $Z\subseteq X_2$ and $(X_1\cup W)\cap X_2=\emptyset$.

Assume that $d_{\rm rNNI}(\cT, \cT')\neq 1$. Then $C_{\cT}(u)-(X_1\cup W\cup Z)$ is non-empty. Let $\ell\in C_{\cT}(u)-(X_1\cup W\cup Z)$, and note that $x_1w|\ell\in \cR(\cT)$, but $x_1w|\ell\not\in \cR(\cT')$. Also, if $X_1$ is a Type~I moving subtree, then either $x_1\ell|z\in \cR(\cT)$ and $x_1z|\ell\in \cR(\cT')$, or $z\ell|x_1\in \cR(\cT)$ and $x_1z|\ell\in \cR(\cT')$, while if $X_1$ is a Type~II moving subtree, then $x_1\ell|z\in \cR(\cT)$ and $\ell z|x_1\in \cR(\cT')$. Thus $X_2\neq W$ as $\{(X\cup \{\rho\})-W, W\}$ is not an agreement forest for $\cT$ and $\cT'$, and so $Z\subseteq X_2$ and $(X_1\cup W)\cap X_2=\emptyset$. Therefore, as $x_1w|\ell\in \cR(\cT)$ but $x_1w|\ell\not\in \cR(\cT')$, it follows that $Z\cup \{\ell\}\subseteq X_2$. But then, as $(X_1\cup W)\cap X_2=\emptyset$, we have $z\ell|x_1\in \cR(\cT)$ and $z\ell|x_1\in \cR(\cT')$, a contradiction. Thus if $d_{\rm rNNI}(\cT, \cT')\neq 1$, then $|M(\cT, \cT')|=1$.

Now assume that $d_{\rm rNNI}(\cT, \cT')=1$. Then
$$C_{\cT}(u)=C_{\cT'}(u')=X_1\cup W\cup Z,$$
and it is easily checked that each of $X_1$, $W$, and $Z$ is a moving subtree for $\cT$ and $\cT'$. Moreover, these are the only moving subtrees for $\cT$ and $\cT'$ as we argued earlier that if $X_2\in M(\cT, \cT')-\{X_1\}$, then either $X_2=W$, or $Z\subseteq X_2$ and $(X_1\cup W)\cap X_2=\emptyset$. Hence if $d_{\rm rNNI}(\cT, \cT')=1$, then $|M(\cT, \cT')|=3$. This completes the proof of the proposition.
\end{proof}

The next two lemmas and Proposition~\ref{unique-subsets} establish properties of hypercubes and graphs that are isomorphic to a hypercube. Let $H$ be a graph, and suppose that $H$ is isomorphic to $Q_k$ for some non-negative integer $k$. Then there is an isomorphism $\varphi$ from the set of vertices of $H$ to the set $B$ of all $k$-bit strings such that if $u$ and $v$ are adjacent vertices in $H$, then $\varphi(u)$ and $\varphi(v)$ differ in exactly one position. Under $\varphi$, the $i$-th bit edge subset of $H$ is the subset of edges whose end vertices differ precisely in the $i$-th position. However, $\varphi$ is not the only such isomorphism. We next show that, regardless of the isomorphism, the collection of subsets of the edge set of $H$ corresponding to the bit edge subsets of $Q_k$ is always the same. The following result can be found, for example, in~\cite{laborde82}.

\begin{lemma}
Let $k$ be a non-negative integer, and let $e$ and $f$ be adjacent edges of $Q_k$. Then $Q_k$ has a unique $4$-cycle containing $e$ and $f$.
\label{4-cycle} 
\end{lemma}

\begin{lemma}
Let $H$ be a graph, and suppose that $\varphi$ is a (graph) isomorphism between $H$ and $Q_k$ for some non-negative integer $k\ge 2$. If $v_1, v_2, v_3, v_4, v_1$ is a $4$-cycle of $H$, then the edges $\{\varphi(v_1), \varphi(v_2)\}$ and $\{\varphi(v_3), \varphi(v_4)\}$ are in the same bit edge subset of $Q_k$.
\label{opposite}
\end{lemma}

\begin{proof}
Let $v_1, v_2, v_3, v_4, v_1$ be a $4$-cycle of $H$. For all $i\in \{1, 2, 3, 4\}$, we have $\varphi(v_i)$ is a $k$-bit string. Since $v_1$ and $v_2$, and $v_1$ and $v_4$ are adjacent, it follows that $\varphi(v_1)$ and $\varphi(v_2)$ differ in precisely one bit, say the $i$-th bit, and $\varphi(v_1)$ and $\varphi(v_4)$ differ in precisely one bit, say the $j$-th bit. Furthermore, as $\varphi(v_2)\neq \varphi(v_4)$, we have $i\neq j$. Also, as $H$ is isomorphic to $Q_k$, the vertices $v_1$ and $v_3$ are not adjacent in $H$, and so $\varphi(v_1)$ and $\varphi(v_3)$ differ in precisely two bits. But $v_3$ is adjacent to $v_2$ and $v_4$, so these two bits are the $i$-th and $j$-th bits. Hence, without loss of generality, we may assume that the $i$-th and $j$-th bits of $\varphi(v_1)$, $\varphi(v_2)$, $\varphi(v_3)$, and $\varphi(v_4)$ are $0\, 0$, $1\, 0$, $1\, 1$, and $0\, 1$, respectively. It now follows that $\{\varphi(v_1), \varphi(v_2)\}$ and $\{\varphi(v_3), \varphi(v_4)\}$ are in the same bit edge subset of $Q_k$.
\end{proof}

\begin{proposition}
Let $H$ be a graph, and suppose that $\varphi_1$ and $\varphi_2$ are both (graph) isomorphisms between $H$ and $Q_k$ for some non-negative integer $k$. Then
$$\{\varphi_1^{-1}(E_1), \varphi_1^{-1}(E_2), \ldots, \varphi_1^{-1}(E_k)\}=\{\varphi_2^{-1}(E_1), \varphi_2^{-1}(E_2), \ldots, \varphi_2^{-1}(E_k)\}$$
where, for all $i\in \{1, 2, \ldots, k\}$, the sets $\varphi_1^{-1}(E_i)$ and $\varphi_2^{-1}(E_i)$ are the subsets of $E(H)$ mapped to $E_i$ under $\varphi_1$ and $\varphi_2$, respectively.
\label{unique-subsets}
\end{proposition}

\begin{proof}
Suppose that $H$ is isomorphic to $Q_k$ for some non-negative integer $k$. If $k\in \{0, 1\}$, then the proposition trivially holds, so assume that $k\ge 2$. Let $v_1$ be a vertex of $H$. Then, under an isomorphism between $H$ and $Q_k$, each of the $k$ edges incident with $v_1$ are assigned to distinct bit edge subsets of $Q_k$. Let $f_1$ be an edge incident with $v_1$ in $H$. We next show that, regardless of the choice of isomorphism between $H$ and $Q_k$, the bit edge subset of $Q_k$ containing the image of $f_1$ is always the same subset of edges of $Q_k$.

Since $H$ is isomorphic to $Q_k$, it follows that $H$ has a \Blue{Hamilton} cycle
$$v_1,\, e_1,\, v_2,\, e_2,\, v_3,\, \ldots,\, v_{2^k},\, e_{2^k},\, v_1$$
starting at $v_1$. Consider \Green{Algorithm~\ref{alg:compute-edge-subset}} which traverses this \Blue{Hamilton} cycle\Green{.} We next show that the image of $F$ is a bit edge subset of $Q_k$. By Lemma~\ref{4-cycle}, \Blue{Line}~\ref{repeat} is well defined, that is, there is a unique edge of $H$ incident with $v_{i+1}$ that is in the unique $4$-cycle containing $f_i$ and $e_i$. Thus, by Lemma~\ref{opposite}, regardless of the choice of isomorphism between $H$ and $Q_k$, the bit edge subset of $Q_k$ that contains the image of $f_1$ also contains the image of $f_2$ and, more generally, the bit edge subset of $Q_k$ that contains the image of $f_i$ also contains the image of $f_{i+1}$. Hence, the bit edge subset of $Q_k$ that contains the image of $f_1$, contains the images of each of the edges in $F$.

\begin{algorithm}
  \caption{\sc Compute Bit Edge Subset}
  \label{alg:compute-edge-subset}

\smallskip

  \KwIn{A graph $H$ that is isomorphic to $Q_k$ for some integer $k\geq 2$, a Hamilton cycle $v_1,\, e_1,\, v_2,\, e_2,\, v_3,\, \ldots,\, v_{2^k},\, e_{2^k},\, v_1$ of $H$ and an edge $f_1$ incident with $v_1$ in $H$.}
  \KwOut{A subset $F$ of the edges of $H$.}

\smallskip

$F \leftarrow \{f_1\}$\;
\For{$i \leftarrow 1$ \KwTo $2^k - 1$}{
  \uIf{$f_i\neq e_i$}{
    Set $f_{i+1}$ to be the edge of $H$ incident with $v_{i+1}$ that is in the unique $4$-cycle containing $f_i$ and $e_i$ \; \label{repeat}
  }
  \Else{
    $f_{i+1} \leftarrow f_i$
  }
  $F \leftarrow F \cup \{f_{i+1}\}$\;
}
\Return{$F$}
\end{algorithm}

Now, by construction, it is easily checked that every vertex of $H$ is incident with exactly one edge in $F$, that is, $F$ is a perfect matching of $H$, and so $|F|=2^{k-1}$. Since a bit edge subset of $Q_k$ has size $2^{k-1}$, it follows that the image of $F$ is a bit edge subset of $Q_k$. Repeating this process for each of the remaining edges incident with $v_1$ establishes the proposition.
\end{proof}

Let $G$ be the rSPR graph of a collection $\cP$ of phylogenetic $X$-trees such that $G\cong Q_k$. Furthermore, let $e=\{\cT,\cT'\}$ be an edge in the $i$-th bit edge subset $E_i$ of $G$ for some $i\in\{1,2,\ldots,k\}$. If $(X',Y')$ is an ordered pair such that $X'\in M(\cT,\cT')$ and $Y'$ is the minimal cluster in $C(\cT)\cap C(\cT')$ that properly contains $X'$, we say that $(X',Y')$ is an {\it ordered pair for} $e$.  Moreover, $(X',Y')$ is said to {\it verify} $E_i$ if $(X',Y')$ is an ordered pair for each edge in $E_i$. Now suppose that $$O=((X_1,Y_1),(X_2,Y_2),\ldots,(X_k,Y_k))$$ is a sequence of distinct ordered pairs such that each $(X_i,Y_i)$ with $i\in\{1,2,\ldots,k\}$ verifies $E_i$. If, for each pair $i,j\in\{1,2,\ldots,k\}$ and $i\ne j$, the two ordered pairs $(X_i,Y_i)$ and $(X_j,Y_j)$ satisfy one of (I)--(III) in the definition of the nested subtree property, then we say that $O$  {\it verifies} the nested subtree property of $G$.

The next two lemmas establish properties of ordered pairs for edges of an rSPR graph. These properties are then used in Proposition~\ref{p:order-does-not-matter} to show that, provided an rSPR graph $G$ has the nested subtree property, any ordered pair that labels an edge $\{\cT,\cT'\}$ with $d_{\rm rNNI}(\cT, \cT')=1$ and verifies its associated bit edge subset can also be used to verify the nested subtree property of $G$.

\begin{lemma}\label{l:help1}
Let $\{\cT,\cT'\}$ be an edge of an rSPR graph $G$. If $(X_1,Y_1)$, $(X_2,Y_2)$, and $(X_3,Y_3)$  are ordered pairs for $e$, then $Y_1=Y_2=Y_3$, $X_1\cup X_2\cup X_3=Y_1$, and the three sets $X_1$, $X_2$, and $X_3$ are pairwise disjoint. 
\end{lemma}

\begin{proof}
By Proposition~\ref{p:one-or-three}, we have $M(\cT,\cT')=\{X_1,X_2,X_3\}$ and $d_{\rm rNNI}(\cT, \cT')=1$. The lemma now follows from the definition of an rNNI move.
\end{proof}

\begin{lemma} \label{l:help2} 
Let $G$ be the rSPR graph for a collection $\cP$ of phylogenetic $X$-trees such that $G\cong Q_k$ for some non-negative integer $k$. Let $(X',Y')$ be an ordered pair that verifies $E_i$ for some $i\in\{1,2,\ldots,k\}$. Then $X'$ and $Y'$ are clusters of each element in $\cP$. 
\end{lemma}

\begin{proof}
Let $\cT$ be a vertex of $G$. Since $G\cong Q_k$, so $E_i$ is a perfect matching of $G$, it follows that $\cT$ is incident with an edge $e\in E_i$. Hence, $X'$ and $Y'$ are both elements in $C(\cT)$.
\end{proof}

\begin{proposition}\label{p:order-does-not-matter}
Let $G$ be the rSPR graph for a collection $\cP$ of phylogenetic $X$-trees such that $G\cong Q_k$ for some non-negative integer $k$. Let $e=\{\cT,\cT'\}$ be an edge of $E_i$ with $i\in\{1,2,\ldots,k\}$ such that $(X_i^1,Y_i^1)$, $(X_i^2,Y_i^2)$, and $(X_i^3,Y_i^3)$ are ordered pairs for $e$.  Suppose that the sequence of ordered pairs $$((X_1,Y_1),(X_2,Y_2),\ldots,(X_k,Y_k))$$ verifies the nested subtree property of $G$ and that $(X_i,Y_i)=(X_i^1,Y_i^1)$. If $(X_i^\ell,Y_i^\ell)$ with $\ell\in\{2,3\}$ verifies $E_i$, then the sequence $$((X_1,Y_1),(X_2,Y_2),\ldots,(X_{i-1},Y_{i-1}),(X^\ell_i,Y^\ell_i),(X_{i+1},Y_{i+1}),\ldots,(X_k,Y_k))$$ also verifies the nested subtree property of $G$. 
\end{proposition}

\begin{proof}
By Proposition~\ref{p:one-or-three}, we first observe that $M(\cT,\cT')=\{X_i^1,X_i^2,X_i^3\}$ and $d_{\rm rNNI}(\cT, \cT')=1$. Now, suppose that $(X_i^\ell,Y_i^\ell)$ with $\ell\in\{2,3\}$ verifies $E_i$. Without loss of generality, we may assume that $\ell=2$. Let $f$ be an edge of $G$ such that $f\in E_j$ with $i\ne j$. We consider three cases depending on which of the three properties in the definition of the nested subtree property the two ordered pairs $(X_i,Y_i)$ and $(X_j,Y_j)$ satisfy. To this end, recall Lemma~\ref{l:help1}, which we will freely use throughout all three cases.

First, suppose that $(X_i,Y_i)$ and $(X_j,Y_j)$ satisfy (I). As $Y_i\cap Y_j=\emptyset$ it follows that $Y_i^2\cap Y_j=\emptyset$. Thus $(X_i^2,Y_i^2)$ and $(X_j,Y_j)$ also satisfy (I).

Second, suppose that $(X_i,Y_i)$ and $(X_j,Y_j)$ satisfy (II). Clearly, if $Y_i\subseteq X_j$, then it immediately follows that $Y_i^2\subseteq X_j$. Hence $(X_i^2,Y_i^2)$ and $(X_j,Y_j)$ also satisfy (II). We may therefore assume that $Y_j\subseteq X_i$. As $Y_j\subseteq X_i\subset Y_i$, we have $Y_j\subset Y_i^2$. Moreover, because $Y_j\subseteq X_i$ and $X_i\cap X_i^2=\emptyset$, it follows that $X_i^2\cap Y_j=\emptyset$. Thus $(X_i^2,Y_i^2)$ and $(X_j,Y_j)$ satisfy (III).

Third, suppose that $(X_i,Y_i)$ and $(X_j,Y_j)$ satisfy (III). Similar to the previous case, if $Y_i\subset Y_j$ and $X_j\cap Y_i=\emptyset$, then $Y_i^2\subset Y_j$ and $X_j\cap Y_i^2=\emptyset$ and, so, $(X_i^2,Y_i^2)$ and $(X_j,Y_j)$ also satisfy (III). Therefore, assume that $Y_j\subset Y_i$ and $X_i\cap Y_j=\emptyset$. If $X_i^2\cap Y_j=\emptyset$, then, as $Y_j\subset Y^2_i$, it follows that $(X_i^2,Y_i^2)$ and $(X_j,Y_j)$ again satisfy (III). On the other hand, if $X_i^2\cap Y_j\ne\emptyset$, we consider $X_i^3$ to complete the argument. Assume that $X_i^3\cap Y_j\ne\emptyset$. Then, as $X^1_i\cup X^2_i\cup X^3_i=Y_i$, we have $Y_j\subseteq X_i^2\cup X_i^3$. As  $d_{\rm rNNI}(\cT, \cT')=1$, each element in $M(\cT,\cT')=\{X_i^1,X_i^2,X_i^3\}$ is a cluster of $\cT$ and $\cT'$, and $X_i^2\cup X_i^3$ is a cluster of at most one of $\cT$ and $\cT'$. But by Lemma~\ref{l:help2}, $Y_j$ is a cluster of $\cT$ and $\cT'$; a contradiction.  Hence $X_i^3\cap Y_j=\emptyset$, and so $Y_j\subset Y_i$. Thus $(X^2_i, Y^2_i)$ and $(X_j, Y_j)$ satisfy (II).

For all three cases, it now follows that $$((X_1,Y_1),(X_2,Y_2),\ldots,(X_{i-1},Y_{i-1}),(X^2_i,Y^2_i),(X_{i+1},Y_{i+1}),\ldots,(X_k,Y_k))$$ verifies the nested subtree property for $G$, thereby establishing the proposition.
\end{proof}

We are now in a position to present the algorithm  {\sc Construct Level-$1$ Network} that constructs a level-$1$ network whose display set is a given set of phylogenetic trees if such a network exists.

\begin{algorithm}[H]
  \caption{{\sc Construct Level-$1$ Network} (Part 1)}
  \label{alg:Construct-Network}
  \KwIn{A collection $\cP$ of phylogenetic $X$-trees.}
  \KwOut{A level-$1$ network $\cN$ on $X$ with $T(\cN)=\cP$ if such a network exists or, otherwise, a statement saying that no such network exists.}

$k \leftarrow \log_2|\cP|$\;\label{alg:start}  

\If{$k$ is not a non-negative integer}{ 
  \Return{``There is no level-$1$ network on $X$ whose display set is $\cP$.''}\; 
}

\If{$k=0$}{
  \Return{the unique element in $\cP$.}
}

Construct the rSPR graph $G$ of $\cP$ and, for each pair $\cT,\cT'\in\cP$ with $d_\rSPR(\cT,\cT')=1$, compute $M(\cT,\cT')$.\;

\If{$G$ is not isomorphic to $Q_k$}{ 
  \Return{``There is no level-$1$ network on $X$ whose display set is $\cP$.''}\; 
}

\For{$i \leftarrow 1$ \KwTo $k$}{
  \uIf{there exists an ordered pair $(X_i,Y_i)$ that verifies $E_i$ and, for each $j\in\{1,2,\ldots,i-1\}$, the pairs $(X_i,Y_i)$ and $(X_j,Y_j)$ satisfy one of (I)--(III) in the definition of the nested subtree property}{
    set $(X_i,Y_i)$ to be such an ordered pair 
  }
  \Else{
      \Return{``There is no level-$1$ network on $X$ whose display set is $\cP$.''}\;\label{alg:correct-decision} 
  }
}

Set $((X_1^0,Y_1^0),\ldots,(X_k^0,Y_k^0))$ to be a permutation of $((X_1,Y_1),\ldots,(X_k,Y_k))$ such that, for each pair $i,j\in\{1,2,\ldots,k\}$ with $i<j$ either $Y_i^0\cap Y_j^0=\emptyset$ or $Y_i^0\subset Y_j^0$.\;\label{alg:start_construction}

$\cP_0 \leftarrow \cP$ and $X_0 \leftarrow X$\;

\For{$i \leftarrow 1$ \KwTo $k$}{
  $X_i \leftarrow (X_{i-1}-Y_i^{i-1})\cup\{y_i\}$\;
  
  Obtain $\cP_i$ from $\cP_{i-1}$ by replacing each $\cT\in\cP_{i-1}$ with the phylogenetic tree on $X_i$ resulting from a subtree reduction on $Y_i^{i-1}$ with replacement leaf $y_i$.\;
  \For{$j \leftarrow 1$ \KwTo $k$}{
    \uIf{$Y_i^{i-1}\subseteq X_j^{i-1}$}{
      $(X_j^i,Y_j^i) \leftarrow ((X_j^{i-1}-Y_i^{i-1})\cup\{y_i\},(Y_j^{i-1}-Y_i^{i-1})\cup\{y_i\})$\;
    }
    \uElseIf{$Y_i^{i-1}\subset Y_j^{i-1}$ and $X_j^{i-1}\cap Y_i^{i-1}=\emptyset$}{
      $(X_j^i,Y_j^i) \leftarrow (X_j^{i-1},(Y_j^{i-1}-Y_i^{i-1})\cup\{y_i\})$\; 
    }
    \Else{
      $(X_j^i,Y_j^i) \leftarrow (X_j^{i-1},Y_j^{i-1})$\;
    }
  }
}

\end{algorithm}

\setcounter{algocf}{1} % set algorithm counter to 1 such that both parts have same algorithm number
\begin{algorithm}[H]
  \caption{{\sc Construct Level-$1$ Network} (Part 2)}
  \setcounter{AlgoLine}{25} % This needs to be adapted if the number of lines in the first part of the algorithm is changed. 

$\cT \leftarrow \text{the unique phylogenetic $X_k$-tree in $\cP_k$}$\;

$\cN_k \leftarrow \cT$\;

$i \leftarrow k$\;
 
\Repeat{$i < 1$}{
  $\cT \leftarrow $ an element in $\cP_{i-1}$ \;
  Obtain $\cN_i'$ from $\cN_i$ by replacing the leaf labelled $y_i$ with $\cT|Y_i^{i-1}$.\;
  
  Set $\cS$ to be an element in $\cP_{i-1}$ such that $\cT|Y_i^{i-1}\ncong \cS|Y_i^{i-1}$.\;
  
  Set $u$ to be the vertex of $\cT$ such that $X_i^{i-1}\subset C_\cT(u)$ and no child $w$ of $u$ in $\cT$ satisfies $X_i^{i-1}\subset C_\cT(w)$.\;
  
   Set $u'$ to be the vertex of $\cS$ such that $X_i^{i-1}\subset C_\cS(u')$ and no child $w$ of $u'$ in $\cS$ satisfies $X_i^{i-1}\subset C_\cS(w)$.\;
   
    \uIf{$(C_\cT(u)-X_i^{i-1})\cap (C_\cS(u')-X_i^{i-1})=\emptyset$ or $C_\cS(u')\subseteq C_\cT(u)$}{
    set $\cN_{i-1}$ to be the network obtained from $\cN_i'$ by subdividing the arc directed into the (unique) vertex whose cluster is $C_{\cS}(u')-X_i^{i-1}$ with a new vertex $v$, subdividing the arc directed into the (unique) vertex whose cluster is $X_i^{i-1}$ with a new vertex $v'$, and adding the new arc $(v,v')$\;
    }
    \ElseIf{$C_\cT(u)\subseteq C_\cS(u')$}{
      set $\cN_{i-1}$ to be the network obtained from $\cN_i'$ by subdividing the arc directed into the (unique) vertex whose cluster is $C_{\cS}(u')$ with a new vertex $v$, subdividing the arc directed into the (unique) vertex whose cluster is $X_i^{i-1}$ with a new vertex $v'$, and adding the new arc $(v,v')$\;
    }
  
  $i \leftarrow i-1$
}

$\cN \leftarrow \cN_0$\; 
\Return{$\cN$}\;\label{alg:end}
\end{algorithm}

\begin{theorem}\label{t:run-time}
Let $\cP$ be a set of phylogenetic $X$-trees. Then {\sc Construct Level-$1$ Network} correctly decides if $\cP$ is the display set of a level-$1$ network on $X$ and, if so, reconstructs such a network. Moreover the running time of the algorithm is $O(|\cP|^2|X|^2)$.
\end{theorem}

\begin{proof}
Let $G$ be the rSPR graph of $\cP$. Suppose that $G\cong Q_k$  for some non-negative integer $k$. Let $\{E_1,E_2,\ldots, E_k\}$ be the partition of the edge set of $G$ such that each $E_i$ with $i\in\{1,2,\ldots,k\}$ corresponds to the $i$-th bit edge subset of $Q_k$. By Proposition~\ref{unique-subsets}, this partition is well defined. Hence, it follows from Theorem~\ref{t:main} that  {\sc Construct Level-$1$ Network} correctly decides whether or not $\cP$ is the display set of a level-$1$ network on $X$. That is, the algorithm completes \Blue{Lines \ref{alg:start}--\ref{alg:correct-decision}} without returning ``There is no level-$1$ network on $X$ whose display set is $\cP$'' if and only if $\cP$ is the display set of a level-$1$ network on $X$. Now suppose that the algorithm completes \Blue{Lines \ref{alg:start}--\ref{alg:correct-decision}} without returning ``There is no level-$1$ network on $X$ whose display set is $\cP$''. \Green{It then follows from the construction given in the inductive proof of Lemma~\ref{main2} that} \Blue{Lines \ref{alg:start_construction}--\ref{alg:end}} of {\sc Construct Level-$1$ Network} correctly reconstruct a level-$1$ network whose display set is $\cP$. 

\Green{
In preparation for the running time analysis, we discuss implementation details and suitable data structures next. For all directed and undirected graphs, \Blue{we use adjacency lists to store arcs (resp.\ edges), and red-black trees to perform binary set operations and comparisons in time at most $O(|X| \log |X|)$. For details about these data structures, we refer the interested reader to~\cite{cormen09}}. Several steps in {\sc Construct Level-$1$ Network} involve finding clusters both in phylogenetic trees and phylogenetic networks. Let $\cN$ be a level-1 network on $X$. Asano et al.~\cite[Theorem 4]{asano12} show how to compute a certain cluster representation of $\cN$ in time $O(|X|)$. Roughly speaking, the authors describe a clever way to number the leaves of $\cN$ such that each cluster can be described by a discrete interval. A similar approach was previously taken by Day~\cite{day85} to obtain efficient representations for clusters in phylogenetic trees. Using this interval-based representation, we can perform each of the following operations in time at most $O(|X|)$, which we will freely use throughout the remainder of the proof.
\begin{enumerate}[(i)]
\item Find a cluster $C_\cN(u)$ of a vertex $u$ in $\cN$.
\item For a given subset $X' \subseteq X$, decide if there is a vertex $u$ in $\cN$ with $C_\cN(u) = X'$ and in the case of existence also find all vertices with this property. 
\item For a given subset $X' \subseteq X$, find a vertex $u$ in $\cN$ with $X' \subset C_\cN(u)$ and no child $w$ of $u$ satisfies $X' \subset C_\cN(w)$. 
\end{enumerate}
Note that (i)--(iii) also apply to trees as every phylogenetic tree is a level-1 network.  
}

It remains to show that the running time of {\sc Construct Level-$1$ Network} is $O(|\cP|^2|X|^2)$. We first bound the number of arcs of a level-$1$ network by a function that only depends linearly on $|X|$. Since any level-$1$ network on $X$ is also tree child, it follows from~\cite{cardona09} that such a network has at most $|X|-1$ reticulations. Hence, as each level-$1$ network $\cN_i$ (resp. $\cN_i'$) with $i\in\{0,1,2,\ldots,k\}$ that is reconstructed in \Blue{Lines \ref{alg:start_construction}--\ref{alg:end}} of the algorithm has at most $k$ reticulations, it follows from~\cite[Lemma 2.1]{mcdiarmid15}  that $\cN_i$ (resp. $\cN_i'$) has at most $$3k+2|X|-2\leq 3(|X|-1)+2|X|-2=5|X|-5$$ arcs. \Green{Now, noting that each of \Green{Lines 1--5, 8, 15, 26--30 and 39--42} takes time $O(1)$, we next detail the running time of the remaining steps.}  

\begin{enumerate}[]
\item {\bf \Blue{Line 6.}} Let $\cT$ and $\cT'$ be two phylogenetic $X$-trees. As mentioned at the end of Section~\ref{sec:properties}, it can be checked in \Green{polynomial} time if $d_\rSPR(\cT,\cT')=1$. \Green{A straightforward way to implement this check is as follows. For each arc $(u,v)$ in $\cT$, let $X_v = C_{\cT}(v)$. Then check if there is a vertex $v'$ in $\cT'$ with $C_{\cT'}(v') = X_v$, $\cT| X_v \cong \cT'| X_v$ and $\cT| (X-X_v) \cong \cT'| (X-X_v)$. Since deciding if two phylogenetic trees are isomorphic can be done in $O(|X|)$, e.g. with the algorithm given by Gusfield~\cite{gusfield91}, the above can be implemented such that the check if $d_\rSPR(\cT,\cT')=1$ takes time $O(|X|^2)$ in total.} Hence, it takes the same time to compute $M(\cT,\cT')$ and time \Green{$O(|\cP|^2|X|^2)$} to reconstruct the rSPR graph $G$ of $\cP$ together with the set of moving subtrees for each edge.\\

\item {\bf \Blue{Line 7.}} Checking if $G$ is isomorphic to $Q_k$ takes time $O(|\cP|\log_2|\cP|)$~\cite{bhat80}.\\

\item {\bf \Blue{Lines 9--13.}} Given a moving subtree $X'$, it takes time $O(|X|)$ to compute the ordered pair $(X',Y')$ \Green{ and testing two such pairs for equality takes time $O(|X|\log |X|)$}.  Now, recall that $|E_i|=2^{k-1}$ for each $i\in\{1,2,\ldots,k\}$. Then, as the number of moving subtrees for two phylogenetic trees is at most three (see Proposition~\ref{p:one-or-three}), it takes time $O(2^{k-1}\Green{|X|\log |X|})$ to compute all ordered pairs that verify $E_i$ for each $i\in\{1,2,\ldots,k\}$. Hence, by Proposition~\ref{p:order-does-not-matter} it takes time $O(2^{k-1}k\Green{|X|\log |X|})$ to compute a sequence $$((X_1,Y_1),(X_2,Y_2),\ldots,(X_k,Y_k))$$ of ordered pairs that pairwise satisfy one of (I)--(III) in the definition of the nested subtree property if $G$ has the nested subtree property. With $k=\log_2|\cP|$, it follows that \Blue{Lines 9--13.} take time \Green{$O((|X|\log |X|)(|\cP|\log |\cP|))$}.\\

\item \Green{{\bf Line 14.} We can obtain the permutation in Line 14 by $O(k^2)$ comparisons that each involve checking if $Y_i\cap Y_j=\emptyset$ or $Y_i\subset Y_j$ in time $O(|X|\log |X|)$. If both of these checks fail, we swap the positions of $(X_i, Y_i)$ and $(X_j, Y_j)$. Hence, Line~14 takes time $O(k^2|X|\log |X|)$ which, with $k=\log_2|\cP|$, is $O((\log|\cP|)^2|X|\log |X|)$.}\\

\item {\bf \Blue{Lines 16--25.}} \Blue{Line 17} takes time \Green{$O(|X|\log |X|)$}, \Blue{Line 18} takes time $O(|\cP||X|)$, and \Green{Lines 19--25} \Green{take} time \Green{$O(k|X|\log |X|)$}. \Green{The outer loop is} executed $k$ times and, so, \Blue{Lines 16--25} take time \Green{$O(k|X|\log |X| +k|\cP||X|+k^2|X|\log |X|)$}. Since $k=\log_2 |\cP|$ that is \Green{$O((|X|\log |X|)(|\cP|\log |\cP|))$}.\\

\item {\bf Lines 29--40.} Each of Lines 31, 33, 34, 36 and 38 takes time $O(|X|)$, each of Lines 35 and 37 takes time $O(|X|\log |X|)$, and Line 32 takes time $O(|\cP||X|)$ since $Y_i^{i-1}$ is a cluster of $\cS$ and $\cT$ by construction and non-isomorphism between two phylogenetic trees can be checked in time $O(|X|)$~\cite{gusfield91}. The loop is executed $k$ times and, so, Lines 29--40 take time $O(k|\cP||X|\log |X|)$, that is again $O((|X|\log |X|)(|\cP|\log |\cP|))$.

\end{enumerate}
It now follows that \Blue{Line 6} is the most time-consuming step and {\sc Construct Level-$1$ Network} takes time \Green{$O(|\cP|^2|X|^2)$} as claimed. This completes the proof of the theorem.
\end{proof}

\noindent \Blue{We remark that Whidden and Matsen~\cite{whidden18} have shown that the rSPR graph for a collection $\cP$ of phylogenetic $X$-trees can be computed in time $O(|\cP||X|^2)$. If their result can be extended to not only deciding if $d_\rSPR(\cT,\cT')=1$ for any pair $\cT,\cT'\in\cP$ but, additionally, to compute the set $M(\cT,\cT')$ in the same time, then the running time of {\sc Construct Level-1~Network} can be improved further since Line 6 is the current bottleneck in the running time analysis (see the proof of Theorem~\ref{t:run-time}).}

The next corollary shows that we cannot only reconstruct a level-$1$ network $\cN$ for a collection $\cP$ of phylogenetic trees  such that $T(\cN)=\cP$ if such a network exists but, in fact, reconstruct all level-$1$ networks that have this property. Suppose that the rSPR graph $G$ of $\cP$ is isomorphic to $Q_k$ for some non-negative integer $k$. Then by iterating over all sequences of ordered pairs $((X_1,Y_1),(X_2,Y_2),\ldots,(X_k,Y_k))$ that verify the nested subtree property of $G$, the next corollary is a consequence of  Lemma~\ref{main2} and Theorem~\ref{t:run-time}.

\begin{corollary}
Let $G$ be the rSPR graph of a collection of phylogenetic $X$-trees such that $G\cong Q_k$ for some non-negative integer $k$. For each $E_i$ with $i\in\{1,2\ldots,k\}$, let $n_i$ the number of ordered pairs that verify $E_i$.  If $G$ has the nested subtree property, then there are $\Pi_{i=1}^k n_i$ level-$1$ networks on $X$ with no trivial reticulation whose display set is $\cP$. Moreover, each such network can be reconstructed in polynomial time.
\end{corollary}

\noindent {\bf Acknowledgements.} We thank the referee for their constructive comments.

\end{document}

%% file: honeycomb.pdf_t
\begin{picture}(0,0)%
\includegraphics{honeycomb.pdf}%
\end{picture}%
\setlength{\unitlength}{3522sp}%
\begingroup\makeatletter\ifx\SetFigFont\undefined%
\gdef\SetFigFont#1#2#3#4#5{%
  \reset@font\fontsize{#1}{#2pt}%
  \fontfamily{#3}\fontseries{#4}\fontshape{#5}%
  \selectfont}%
\fi\endgroup%
\begin{picture}(4382,1729)(4290,-4274)
\put(4321,-3841){\makebox(0,0)[b]{\smash{{\SetFigFont{8}{9.6}{\rmdefault}{\mddefault}{\updefault}$1$}}}}
\put(4951,-3931){\makebox(0,0)[b]{\smash{{\SetFigFont{8}{9.6}{\rmdefault}{\mddefault}{\updefault}$2$}}}}
\put(5311,-3931){\makebox(0,0)[b]{\smash{{\SetFigFont{8}{9.6}{\rmdefault}{\mddefault}{\updefault}$3$}}}}
\put(5941,-3841){\makebox(0,0)[b]{\smash{{\SetFigFont{8}{9.6}{\rmdefault}{\mddefault}{\updefault}$4$}}}}
\put(6301,-3301){\makebox(0,0)[b]{\smash{{\SetFigFont{8}{9.6}{\rmdefault}{\mddefault}{\updefault}$1$}}}}
\put(7741,-3301){\makebox(0,0)[b]{\smash{{\SetFigFont{8}{9.6}{\rmdefault}{\mddefault}{\updefault}$1$}}}}
\put(8101,-3301){\makebox(0,0)[b]{\smash{{\SetFigFont{8}{9.6}{\rmdefault}{\mddefault}{\updefault}$1$}}}}
\put(6661,-3931){\makebox(0,0)[b]{\smash{{\SetFigFont{8}{9.6}{\rmdefault}{\mddefault}{\updefault}$1$}}}}
\put(7741,-3931){\makebox(0,0)[b]{\smash{{\SetFigFont{8}{9.6}{\rmdefault}{\mddefault}{\updefault}$1$}}}}
\put(6481,-3301){\makebox(0,0)[b]{\smash{{\SetFigFont{8}{9.6}{\rmdefault}{\mddefault}{\updefault}$2$}}}}
\put(7561,-3301){\makebox(0,0)[b]{\smash{{\SetFigFont{8}{9.6}{\rmdefault}{\mddefault}{\updefault}$2$}}}}
\put(8281,-3301){\makebox(0,0)[b]{\smash{{\SetFigFont{8}{9.6}{\rmdefault}{\mddefault}{\updefault}$2$}}}}
\put(6301,-3931){\makebox(0,0)[b]{\smash{{\SetFigFont{8}{9.6}{\rmdefault}{\mddefault}{\updefault}$2$}}}}
\put(7201,-3931){\makebox(0,0)[b]{\smash{{\SetFigFont{8}{9.6}{\rmdefault}{\mddefault}{\updefault}$2$}}}}
\put(6661,-3301){\makebox(0,0)[b]{\smash{{\SetFigFont{8}{9.6}{\rmdefault}{\mddefault}{\updefault}$3$}}}}
\put(7381,-3301){\makebox(0,0)[b]{\smash{{\SetFigFont{8}{9.6}{\rmdefault}{\mddefault}{\updefault}$3$}}}}
\put(8461,-3301){\makebox(0,0)[b]{\smash{{\SetFigFont{8}{9.6}{\rmdefault}{\mddefault}{\updefault}$3$}}}}
\put(6481,-3931){\makebox(0,0)[b]{\smash{{\SetFigFont{8}{9.6}{\rmdefault}{\mddefault}{\updefault}$3$}}}}
\put(7381,-3931){\makebox(0,0)[b]{\smash{{\SetFigFont{8}{9.6}{\rmdefault}{\mddefault}{\updefault}$3$}}}}
\put(6841,-3301){\makebox(0,0)[b]{\smash{{\SetFigFont{8}{9.6}{\rmdefault}{\mddefault}{\updefault}$4$}}}}
\put(7201,-3301){\makebox(0,0)[b]{\smash{{\SetFigFont{8}{9.6}{\rmdefault}{\mddefault}{\updefault}$4$}}}}
\put(8641,-3301){\makebox(0,0)[b]{\smash{{\SetFigFont{8}{9.6}{\rmdefault}{\mddefault}{\updefault}$4$}}}}
\put(6841,-3931){\makebox(0,0)[b]{\smash{{\SetFigFont{8}{9.6}{\rmdefault}{\mddefault}{\updefault}$4$}}}}
\put(7561,-3931){\makebox(0,0)[b]{\smash{{\SetFigFont{8}{9.6}{\rmdefault}{\mddefault}{\updefault}$4$}}}}
\put(7471,-2716){\makebox(0,0)[b]{\smash{{\SetFigFont{9}{10.8}{\rmdefault}{\mddefault}{\updefault}{\color[rgb]{0,0,0}$\cT_2$}%
}}}}
\put(8371,-2716){\makebox(0,0)[b]{\smash{{\SetFigFont{9}{10.8}{\rmdefault}{\mddefault}{\updefault}{\color[rgb]{0,0,0}$\cT_3$}%
}}}}
\put(6571,-4201){\makebox(0,0)[b]{\smash{{\SetFigFont{9}{10.8}{\rmdefault}{\mddefault}{\updefault}{\color[rgb]{0,0,0}$\cT_4$}%
}}}}
\put(7471,-4201){\makebox(0,0)[b]{\smash{{\SetFigFont{9}{10.8}{\rmdefault}{\mddefault}{\updefault}{\color[rgb]{0,0,0}$\cT_5$}%
}}}}
\put(6571,-2716){\makebox(0,0)[b]{\smash{{\SetFigFont{9}{10.8}{\rmdefault}{\mddefault}{\updefault}{\color[rgb]{0,0,0}$\cT_1$}%
}}}}
\end{picture}%

%% file: rSPR-graph.pdf_t
\begin{picture}(0,0)%
\includegraphics{rSPR-graph.pdf}%
\end{picture}%
\setlength{\unitlength}{3522sp}%
\begingroup\makeatletter\ifx\SetFigFont\undefined%
\gdef\SetFigFont#1#2#3#4#5{%
  \reset@font\fontsize{#1}{#2pt}%
  \fontfamily{#3}\fontseries{#4}\fontshape{#5}%
  \selectfont}%
\fi\endgroup%
\begin{picture}(5941,4162)(4290,-6371)
\put(8011,-5101){\makebox(0,0)[b]{\smash{{\SetFigFont{8}{9.6}{\rmdefault}{\mddefault}{\updefault}$1$}}}}
\put(8056,-3301){\makebox(0,0)[b]{\smash{{\SetFigFont{9}{10.8}{\rmdefault}{\mddefault}{\updefault}$G$}}}}
\put(4861,-3616){\makebox(0,0)[b]{\smash{{\SetFigFont{9}{10.8}{\rmdefault}{\mddefault}{\updefault}$2$}}}}
\put(10216,-3301){\makebox(0,0)[b]{\smash{{\SetFigFont{9}{10.8}{\rmdefault}{\mddefault}{\updefault}$Q_2$}}}}
\put(8551,-2896){\makebox(0,0)[b]{\smash{{\SetFigFont{8}{9.6}{\rmdefault}{\mddefault}{\updefault}$00$}}}}
\put(9901,-2896){\makebox(0,0)[b]{\smash{{\SetFigFont{8}{9.6}{\rmdefault}{\mddefault}{\updefault}$01$}}}}
\put(8551,-3751){\makebox(0,0)[b]{\smash{{\SetFigFont{8}{9.6}{\rmdefault}{\mddefault}{\updefault}$10$}}}}
\put(9901,-3751){\makebox(0,0)[b]{\smash{{\SetFigFont{8}{9.6}{\rmdefault}{\mddefault}{\updefault}$11$}}}}
\put(5491,-3166){\makebox(0,0)[lb]{\smash{{\SetFigFont{6}{7.2}{\rmdefault}{\mddefault}{\updefault}$1$}}}}
\put(4906,-3301){\makebox(0,0)[lb]{\smash{{\SetFigFont{6}{7.2}{\rmdefault}{\mddefault}{\updefault}$v_1$}}}}
\put(5446,-3301){\makebox(0,0)[lb]{\smash{{\SetFigFont{6}{7.2}{\rmdefault}{\mddefault}{\updefault}$v_2$}}}}
\put(5266,-3166){\makebox(0,0)[rb]{\smash{{\SetFigFont{6}{7.2}{\rmdefault}{\mddefault}{\updefault}$0$}}}}
\put(4996,-3166){\makebox(0,0)[lb]{\smash{{\SetFigFont{6}{7.2}{\rmdefault}{\mddefault}{\updefault}$1$}}}}
\put(4771,-3166){\makebox(0,0)[rb]{\smash{{\SetFigFont{6}{7.2}{\rmdefault}{\mddefault}{\updefault}$0$}}}}
\put(4321,-3616){\makebox(0,0)[b]{\smash{{\SetFigFont{9}{10.8}{\rmdefault}{\mddefault}{\updefault}$1$}}}}
\put(5941,-3616){\makebox(0,0)[b]{\smash{{\SetFigFont{9}{10.8}{\rmdefault}{\mddefault}{\updefault}$4$}}}}
\put(5401,-3616){\makebox(0,0)[b]{\smash{{\SetFigFont{9}{10.8}{\rmdefault}{\mddefault}{\updefault}$3$}}}}
\put(5131,-3931){\makebox(0,0)[b]{\smash{{\SetFigFont{9}{10.8}{\rmdefault}{\mddefault}{\updefault}$\cN$}}}}
\put(7651,-6316){\makebox(0,0)[b]{\smash{{\SetFigFont{8}{9.6}{\rmdefault}{\mddefault}{\updefault}$3$}}}}
\put(6661,-6316){\makebox(0,0)[b]{\smash{{\SetFigFont{8}{9.6}{\rmdefault}{\mddefault}{\updefault}$4$}}}}
\put(6121,-6316){\makebox(0,0)[b]{\smash{{\SetFigFont{8}{9.6}{\rmdefault}{\mddefault}{\updefault}$2$}}}}
\put(6301,-6316){\makebox(0,0)[b]{\smash{{\SetFigFont{8}{9.6}{\rmdefault}{\mddefault}{\updefault}$3$}}}}
\put(6481,-6316){\makebox(0,0)[b]{\smash{{\SetFigFont{8}{9.6}{\rmdefault}{\mddefault}{\updefault}$1$}}}}
\put(7471,-6316){\makebox(0,0)[b]{\smash{{\SetFigFont{8}{9.6}{\rmdefault}{\mddefault}{\updefault}$2$}}}}
\put(7831,-6316){\makebox(0,0)[b]{\smash{{\SetFigFont{8}{9.6}{\rmdefault}{\mddefault}{\updefault}$4$}}}}
\put(8011,-6316){\makebox(0,0)[b]{\smash{{\SetFigFont{8}{9.6}{\rmdefault}{\mddefault}{\updefault}$1$}}}}
\put(8056,-5551){\makebox(0,0)[b]{\smash{{\SetFigFont{9}{10.8}{\rmdefault}{\mddefault}{\updefault}$G'$}}}}
\put(5041,-6136){\makebox(0,0)[b]{\smash{{\SetFigFont{9}{10.8}{\rmdefault}{\mddefault}{\updefault}$\cN'$}}}}
\put(4501,-5821){\makebox(0,0)[b]{\smash{{\SetFigFont{9}{10.8}{\rmdefault}{\mddefault}{\updefault}$1$}}}}
\put(4861,-5821){\makebox(0,0)[b]{\smash{{\SetFigFont{9}{10.8}{\rmdefault}{\mddefault}{\updefault}$2$}}}}
\put(5221,-5821){\makebox(0,0)[b]{\smash{{\SetFigFont{9}{10.8}{\rmdefault}{\mddefault}{\updefault}$3$}}}}
\put(5581,-5821){\makebox(0,0)[b]{\smash{{\SetFigFont{9}{10.8}{\rmdefault}{\mddefault}{\updefault}$4$}}}}
\put(7471,-5101){\makebox(0,0)[b]{\smash{{\SetFigFont{8}{9.6}{\rmdefault}{\mddefault}{\updefault}$4$}}}}
\put(7651,-5101){\makebox(0,0)[b]{\smash{{\SetFigFont{8}{9.6}{\rmdefault}{\mddefault}{\updefault}$3$}}}}
\put(7831,-5101){\makebox(0,0)[b]{\smash{{\SetFigFont{8}{9.6}{\rmdefault}{\mddefault}{\updefault}$2$}}}}
\put(6121,-2896){\makebox(0,0)[b]{\smash{{\SetFigFont{8}{9.6}{\rmdefault}{\mddefault}{\updefault}$1$}}}}
\put(6301,-2896){\makebox(0,0)[b]{\smash{{\SetFigFont{8}{9.6}{\rmdefault}{\mddefault}{\updefault}$2$}}}}
\put(6481,-2896){\makebox(0,0)[b]{\smash{{\SetFigFont{8}{9.6}{\rmdefault}{\mddefault}{\updefault}$3$}}}}
\put(6661,-2896){\makebox(0,0)[b]{\smash{{\SetFigFont{8}{9.6}{\rmdefault}{\mddefault}{\updefault}$4$}}}}
\put(6391,-2356){\makebox(0,0)[b]{\smash{{\SetFigFont{8}{9.6}{\rmdefault}{\mddefault}{\updefault}$\cT_{00}$}}}}
\put(7651,-2896){\makebox(0,0)[b]{\smash{{\SetFigFont{8}{9.6}{\rmdefault}{\mddefault}{\updefault}$2$}}}}
\put(7831,-2896){\makebox(0,0)[b]{\smash{{\SetFigFont{8}{9.6}{\rmdefault}{\mddefault}{\updefault}$3$}}}}
\put(8011,-2896){\makebox(0,0)[b]{\smash{{\SetFigFont{8}{9.6}{\rmdefault}{\mddefault}{\updefault}$4$}}}}
\put(7741,-2356){\makebox(0,0)[b]{\smash{{\SetFigFont{8}{9.6}{\rmdefault}{\mddefault}{\updefault}$\cT_{01}$}}}}
\put(7471,-2896){\makebox(0,0)[b]{\smash{{\SetFigFont{8}{9.6}{\rmdefault}{\mddefault}{\updefault}$1$}}}}
\put(7471,-4111){\makebox(0,0)[b]{\smash{{\SetFigFont{8}{9.6}{\rmdefault}{\mddefault}{\updefault}$4$}}}}
\put(7651,-4111){\makebox(0,0)[b]{\smash{{\SetFigFont{8}{9.6}{\rmdefault}{\mddefault}{\updefault}$3$}}}}
\put(7831,-4111){\makebox(0,0)[b]{\smash{{\SetFigFont{8}{9.6}{\rmdefault}{\mddefault}{\updefault}$2$}}}}
\put(8011,-4111){\makebox(0,0)[b]{\smash{{\SetFigFont{8}{9.6}{\rmdefault}{\mddefault}{\updefault}$1$}}}}
\put(7741,-4291){\makebox(0,0)[b]{\smash{{\SetFigFont{8}{9.6}{\rmdefault}{\mddefault}{\updefault}$\cT_{11}$}}}}
\put(6121,-4111){\makebox(0,0)[b]{\smash{{\SetFigFont{8}{9.6}{\rmdefault}{\mddefault}{\updefault}$1$}}}}
\put(6661,-4111){\makebox(0,0)[b]{\smash{{\SetFigFont{8}{9.6}{\rmdefault}{\mddefault}{\updefault}$4$}}}}
\put(6301,-4111){\makebox(0,0)[b]{\smash{{\SetFigFont{8}{9.6}{\rmdefault}{\mddefault}{\updefault}$3$}}}}
\put(6481,-4111){\makebox(0,0)[b]{\smash{{\SetFigFont{8}{9.6}{\rmdefault}{\mddefault}{\updefault}$2$}}}}
\put(6391,-4291){\makebox(0,0)[b]{\smash{{\SetFigFont{8}{9.6}{\rmdefault}{\mddefault}{\updefault}$\cT_{10}$}}}}
\put(6121,-5101){\makebox(0,0)[b]{\smash{{\SetFigFont{8}{9.6}{\rmdefault}{\mddefault}{\updefault}$1$}}}}
\put(6301,-5101){\makebox(0,0)[b]{\smash{{\SetFigFont{8}{9.6}{\rmdefault}{\mddefault}{\updefault}$2$}}}}
\put(6481,-5101){\makebox(0,0)[b]{\smash{{\SetFigFont{8}{9.6}{\rmdefault}{\mddefault}{\updefault}$3$}}}}
\put(6661,-5101){\makebox(0,0)[b]{\smash{{\SetFigFont{8}{9.6}{\rmdefault}{\mddefault}{\updefault}$4$}}}}
\end{picture}%

%% file: nested-subtree-prop.pdf_t
\begin{picture}(0,0)%
\includegraphics{nested-subtree-prop.pdf}%
\end{picture}%
\setlength{\unitlength}{3522sp}%
\begingroup\makeatletter\ifx\SetFigFont\undefined%
\gdef\SetFigFont#1#2#3#4#5{%
  \reset@font\fontsize{#1}{#2pt}%
  \fontfamily{#3}\fontseries{#4}\fontshape{#5}%
  \selectfont}%
\fi\endgroup%
\begin{picture}(3618,2542)(4402,-3671)
\put(7606,-3616){\makebox(0,0)[b]{\smash{{\SetFigFont{8}{9.6}{\rmdefault}{\mddefault}{\updefault}(III)}}}}
\put(7336,-1771){\makebox(0,0)[lb]{\smash{{\SetFigFont{8}{9.6}{\rmdefault}{\mddefault}{\updefault}$u_j$}}}}
\put(7336,-2401){\makebox(0,0)[lb]{\smash{{\SetFigFont{8}{9.6}{\rmdefault}{\mddefault}{\updefault}$v_j$}}}}
\put(7876,-2491){\makebox(0,0)[lb]{\smash{{\SetFigFont{8}{9.6}{\rmdefault}{\mddefault}{\updefault}$u_i$}}}}
\put(7876,-3121){\makebox(0,0)[lb]{\smash{{\SetFigFont{8}{9.6}{\rmdefault}{\mddefault}{\updefault}$v_i$}}}}
\put(4636,-3121){\makebox(0,0)[lb]{\smash{{\SetFigFont{8}{9.6}{\rmdefault}{\mddefault}{\updefault}$v_i$}}}}
\put(5356,-3121){\makebox(0,0)[lb]{\smash{{\SetFigFont{8}{9.6}{\rmdefault}{\mddefault}{\updefault}$v_j$}}}}
\put(5389,-2516){\makebox(0,0)[lb]{\smash{{\SetFigFont{8}{9.6}{\rmdefault}{\mddefault}{\updefault}$u_j$}}}}
\put(4683,-2530){\makebox(0,0)[lb]{\smash{{\SetFigFont{8}{9.6}{\rmdefault}{\mddefault}{\updefault}$u_i$}}}}
\put(4951,-3616){\makebox(0,0)[b]{\smash{{\SetFigFont{8}{9.6}{\rmdefault}{\mddefault}{\updefault}(I)}}}}
\put(6346,-1501){\makebox(0,0)[lb]{\smash{{\SetFigFont{8}{9.6}{\rmdefault}{\mddefault}{\updefault}$u_j$}}}}
\put(6346,-2131){\makebox(0,0)[lb]{\smash{{\SetFigFont{8}{9.6}{\rmdefault}{\mddefault}{\updefault}$v_j$}}}}
\put(6346,-2491){\makebox(0,0)[lb]{\smash{{\SetFigFont{8}{9.6}{\rmdefault}{\mddefault}{\updefault}$u_i$}}}}
\put(6301,-3616){\makebox(0,0)[b]{\smash{{\SetFigFont{8}{9.6}{\rmdefault}{\mddefault}{\updefault}(II)}}}}
\put(6346,-3121){\makebox(0,0)[lb]{\smash{{\SetFigFont{8}{9.6}{\rmdefault}{\mddefault}{\updefault}$v_i$}}}}
\end{picture}%

%% file: level-1.pdf_t
\begin{picture}(0,0)%
\includegraphics{level-1.pdf}%
\end{picture}%
\setlength{\unitlength}{3522sp}%
\begingroup\makeatletter\ifx\SetFigFont\undefined%
\gdef\SetFigFont#1#2#3#4#5{%
  \reset@font\fontsize{#1}{#2pt}%
  \fontfamily{#3}\fontseries{#4}\fontshape{#5}%
  \selectfont}%
\fi\endgroup%
\begin{picture}(6150,2516)(3721,-6686)
\put(4179,-6353){\makebox(0,0)[b]{\smash{{\SetFigFont{9}{10.8}{\rmdefault}{\mddefault}{\updefault}$2$}}}}
\put(3999,-6266){\makebox(0,0)[b]{\smash{{\SetFigFont{9}{10.8}{\rmdefault}{\mddefault}{\updefault}$3$}}}}
\put(3899,-5987){\makebox(0,0)[b]{\smash{{\SetFigFont{9}{10.8}{\rmdefault}{\mddefault}{\updefault}$4$}}}}
\put(4462,-5891){\makebox(0,0)[b]{\smash{{\SetFigFont{6}{7.2}{\rmdefault}{\mddefault}{\updefault}$v_2$}}}}
\put(4456,-5236){\makebox(0,0)[b]{\smash{{\SetFigFont{6}{7.2}{\rmdefault}{\mddefault}{\updefault}$u_2$}}}}
\put(5806,-5011){\makebox(0,0)[b]{\smash{{\SetFigFont{8}{9.6}{\rmdefault}{\mddefault}{\updefault}$\cT_1$}}}}
\put(7156,-5911){\makebox(0,0)[b]{\smash{{\SetFigFont{8}{9.6}{\rmdefault}{\mddefault}{\updefault}$\cT_4$}}}}
\put(5806,-5911){\makebox(0,0)[b]{\smash{{\SetFigFont{8}{9.6}{\rmdefault}{\mddefault}{\updefault}$\cT_3$}}}}
\put(7156,-5011){\makebox(0,0)[b]{\smash{{\SetFigFont{8}{9.6}{\rmdefault}{\mddefault}{\updefault}$\cT_2$}}}}
\put(7516,-5461){\makebox(0,0)[b]{\smash{{\SetFigFont{9}{10.8}{\rmdefault}{\mddefault}{\updefault}$G$}}}}
\put(4819,-5026){\makebox(0,0)[b]{\smash{{\SetFigFont{9}{10.8}{\rmdefault}{\mddefault}{\updefault}$6$}}}}
\put(4086,-5033){\makebox(0,0)[b]{\smash{{\SetFigFont{9}{10.8}{\rmdefault}{\mddefault}{\updefault}$5$}}}}
\put(4480,-4923){\makebox(0,0)[lb]{\smash{{\SetFigFont{6}{7.2}{\rmdefault}{\mddefault}{\updefault}$v_1$}}}}
\put(3736,-5461){\makebox(0,0)[b]{\smash{{\SetFigFont{9}{10.8}{\rmdefault}{\mddefault}{\updefault}$\cN$}}}}
\put(4501,-4471){\makebox(0,0)[lb]{\smash{{\SetFigFont{6}{7.2}{\rmdefault}{\mddefault}{\updefault}$u_1$}}}}
\put(8863,-6489){\makebox(0,0)[b]{\smash{{\SetFigFont{9}{10.8}{\rmdefault}{\mddefault}{\updefault}$1$}}}}
\put(8589,-6443){\makebox(0,0)[b]{\smash{{\SetFigFont{9}{10.8}{\rmdefault}{\mddefault}{\updefault}$2$}}}}
\put(8409,-6356){\makebox(0,0)[b]{\smash{{\SetFigFont{9}{10.8}{\rmdefault}{\mddefault}{\updefault}$3$}}}}
\put(8309,-6077){\makebox(0,0)[b]{\smash{{\SetFigFont{9}{10.8}{\rmdefault}{\mddefault}{\updefault}$4$}}}}
\put(9856,-5461){\makebox(0,0)[b]{\smash{{\SetFigFont{9}{10.8}{\rmdefault}{\mddefault}{\updefault}$\cN'$}}}}
\put(9361,-5416){\makebox(0,0)[b]{\smash{{\SetFigFont{9}{10.8}{\rmdefault}{\mddefault}{\updefault}$5$}}}}
\put(8911,-5056){\makebox(0,0)[b]{\smash{{\SetFigFont{9}{10.8}{\rmdefault}{\mddefault}{\updefault}$6$}}}}
\put(6706,-4831){\makebox(0,0)[b]{\smash{{\SetFigFont{8}{9.6}{\rmdefault}{\mddefault}{\updefault}$1$}}}}
\put(6886,-4831){\makebox(0,0)[b]{\smash{{\SetFigFont{8}{9.6}{\rmdefault}{\mddefault}{\updefault}$2$}}}}
\put(7066,-4831){\makebox(0,0)[b]{\smash{{\SetFigFont{8}{9.6}{\rmdefault}{\mddefault}{\updefault}$3$}}}}
\put(7246,-4831){\makebox(0,0)[b]{\smash{{\SetFigFont{8}{9.6}{\rmdefault}{\mddefault}{\updefault}$4$}}}}
\put(7426,-4831){\makebox(0,0)[b]{\smash{{\SetFigFont{8}{9.6}{\rmdefault}{\mddefault}{\updefault}$6$}}}}
\put(7606,-4831){\makebox(0,0)[b]{\smash{{\SetFigFont{8}{9.6}{\rmdefault}{\mddefault}{\updefault}$5$}}}}
\put(5356,-4831){\makebox(0,0)[b]{\smash{{\SetFigFont{8}{9.6}{\rmdefault}{\mddefault}{\updefault}$1$}}}}
\put(5536,-4831){\makebox(0,0)[b]{\smash{{\SetFigFont{8}{9.6}{\rmdefault}{\mddefault}{\updefault}$2$}}}}
\put(5716,-4831){\makebox(0,0)[b]{\smash{{\SetFigFont{8}{9.6}{\rmdefault}{\mddefault}{\updefault}$3$}}}}
\put(5896,-4831){\makebox(0,0)[b]{\smash{{\SetFigFont{8}{9.6}{\rmdefault}{\mddefault}{\updefault}$4$}}}}
\put(6076,-4831){\makebox(0,0)[b]{\smash{{\SetFigFont{8}{9.6}{\rmdefault}{\mddefault}{\updefault}$5$}}}}
\put(6256,-4831){\makebox(0,0)[b]{\smash{{\SetFigFont{8}{9.6}{\rmdefault}{\mddefault}{\updefault}$6$}}}}
\put(6076,-6631){\makebox(0,0)[b]{\smash{{\SetFigFont{8}{9.6}{\rmdefault}{\mddefault}{\updefault}$5$}}}}
\put(6256,-6631){\makebox(0,0)[b]{\smash{{\SetFigFont{8}{9.6}{\rmdefault}{\mddefault}{\updefault}$6$}}}}
\put(5356,-6631){\makebox(0,0)[b]{\smash{{\SetFigFont{8}{9.6}{\rmdefault}{\mddefault}{\updefault}$2$}}}}
\put(5536,-6631){\makebox(0,0)[b]{\smash{{\SetFigFont{8}{9.6}{\rmdefault}{\mddefault}{\updefault}$3$}}}}
\put(5716,-6631){\makebox(0,0)[b]{\smash{{\SetFigFont{8}{9.6}{\rmdefault}{\mddefault}{\updefault}$4$}}}}
\put(5896,-6631){\makebox(0,0)[b]{\smash{{\SetFigFont{8}{9.6}{\rmdefault}{\mddefault}{\updefault}$1$}}}}
\put(6751,-6631){\makebox(0,0)[b]{\smash{{\SetFigFont{8}{9.6}{\rmdefault}{\mddefault}{\updefault}$2$}}}}
\put(6931,-6631){\makebox(0,0)[b]{\smash{{\SetFigFont{8}{9.6}{\rmdefault}{\mddefault}{\updefault}$3$}}}}
\put(7111,-6631){\makebox(0,0)[b]{\smash{{\SetFigFont{8}{9.6}{\rmdefault}{\mddefault}{\updefault}$4$}}}}
\put(7291,-6631){\makebox(0,0)[b]{\smash{{\SetFigFont{8}{9.6}{\rmdefault}{\mddefault}{\updefault}$1$}}}}
\put(7651,-6631){\makebox(0,0)[b]{\smash{{\SetFigFont{8}{9.6}{\rmdefault}{\mddefault}{\updefault}$5$}}}}
\put(7471,-6631){\makebox(0,0)[b]{\smash{{\SetFigFont{8}{9.6}{\rmdefault}{\mddefault}{\updefault}$6$}}}}
\put(4453,-6399){\makebox(0,0)[b]{\smash{{\SetFigFont{9}{10.8}{\rmdefault}{\mddefault}{\updefault}$1$}}}}
\end{picture}%

%% file: subtree-reduction.pdf_t
\begin{picture}(0,0)%
\includegraphics{subtree-reduction.pdf}%
\end{picture}%
\setlength{\unitlength}{3522sp}%
\begingroup\makeatletter\ifx\SetFigFont\undefined%
\gdef\SetFigFont#1#2#3#4#5{%
  \reset@font\fontsize{#1}{#2pt}%
  \fontfamily{#3}\fontseries{#4}\fontshape{#5}%
  \selectfont}%
\fi\endgroup%
\begin{picture}(1329,1303)(4309,-1871)
\put(5356,-691){\makebox(0,0)[b]{\smash{{\SetFigFont{8}{9.6}{\rmdefault}{\mddefault}{\updefault}$\cT'$}}}}
\put(5356,-1681){\makebox(0,0)[b]{\smash{{\SetFigFont{8}{9.6}{\rmdefault}{\mddefault}{\updefault}$x'$}}}}
\put(4591,-1816){\makebox(0,0)[b]{\smash{{\SetFigFont{8}{9.6}{\rmdefault}{\mddefault}{\updefault}$X'$}}}}
\put(4591,-691){\makebox(0,0)[b]{\smash{{\SetFigFont{8}{9.6}{\rmdefault}{\mddefault}{\updefault}$\cT$}}}}
\end{picture}%

%% file: G_1.pdf_t
\begin{picture}(0,0)%
\includegraphics{G_1.pdf}%
\end{picture}%
\setlength{\unitlength}{3522sp}%
\begingroup\makeatletter\ifx\SetFigFont\undefined%
\gdef\SetFigFont#1#2#3#4#5{%
  \reset@font\fontsize{#1}{#2pt}%
  \fontfamily{#3}\fontseries{#4}\fontshape{#5}%
  \selectfont}%
\fi\endgroup%
\begin{picture}(6189,5737)(3319,-4841)
\put(6076,-1096){\makebox(0,0)[lb]{\smash{{\SetFigFont{8}{9.6}{\rmdefault}{\mddefault}{\updefault}$X_3$}}}}
\put(6256,389){\makebox(0,0)[lb]{\smash{{\SetFigFont{8}{9.6}{\rmdefault}{\mddefault}{\updefault}$v_1$}}}}
\put(6976,299){\makebox(0,0)[lb]{\smash{{\SetFigFont{8}{9.6}{\rmdefault}{\mddefault}{\updefault}$v_2$}}}}
\put(6031,-106){\makebox(0,0)[b]{\smash{{\SetFigFont{8}{9.6}{\rmdefault}{\mddefault}{\updefault}$u_3$}}}}
\put(6661,749){\makebox(0,0)[b]{\smash{{\SetFigFont{8}{9.6}{\rmdefault}{\mddefault}{\updefault}$\cN$}}}}
\put(4276,389){\makebox(0,0)[lb]{\smash{{\SetFigFont{8}{9.6}{\rmdefault}{\mddefault}{\updefault}$v_1$}}}}
\put(4996,299){\makebox(0,0)[lb]{\smash{{\SetFigFont{8}{9.6}{\rmdefault}{\mddefault}{\updefault}$v_2$}}}}
\put(4681,749){\makebox(0,0)[b]{\smash{{\SetFigFont{8}{9.6}{\rmdefault}{\mddefault}{\updefault}$\cN_1$}}}}
\put(8236,389){\makebox(0,0)[lb]{\smash{{\SetFigFont{8}{9.6}{\rmdefault}{\mddefault}{\updefault}$v_1$}}}}
\put(8956,299){\makebox(0,0)[lb]{\smash{{\SetFigFont{8}{9.6}{\rmdefault}{\mddefault}{\updefault}$v_2$}}}}
\put(8641,749){\makebox(0,0)[b]{\smash{{\SetFigFont{8}{9.6}{\rmdefault}{\mddefault}{\updefault}$\cN_2$}}}}
\put(5806,-691){\makebox(0,0)[rb]{\smash{{\SetFigFont{8}{9.6}{\rmdefault}{\mddefault}{\updefault}$p_1$}}}}
\put(6301,-691){\makebox(0,0)[lb]{\smash{{\SetFigFont{8}{9.6}{\rmdefault}{\mddefault}{\updefault}$p_2$}}}}
\put(4321,-1726){\makebox(0,0)[b]{\smash{{\SetFigFont{8}{9.6}{\rmdefault}{\mddefault}{\updefault}$G_1\cong Q_2$}}}}
\put(8281,-1771){\makebox(0,0)[b]{\smash{{\SetFigFont{8}{9.6}{\rmdefault}{\mddefault}{\updefault}$G_2\cong Q_2$}}}}
\put(6436,-3616){\makebox(0,0)[b]{\smash{{\SetFigFont{8}{9.6}{\rmdefault}{\mddefault}{\updefault}rSPR}}}}
\put(4321,-2221){\makebox(0,0)[b]{\smash{{\SetFigFont{8}{9.6}{\rmdefault}{\mddefault}{\updefault}$(X_1,Y_1)$}}}}
\put(4321,-2491){\makebox(0,0)[b]{\smash{{\SetFigFont{8}{9.6}{\rmdefault}{\mddefault}{\updefault}$(X_1,Y_1)$}}}}
\put(8281,-2491){\makebox(0,0)[b]{\smash{{\SetFigFont{8}{9.6}{\rmdefault}{\mddefault}{\updefault}$(X_1,Y_1)$}}}}
\put(8281,-2221){\makebox(0,0)[b]{\smash{{\SetFigFont{8}{9.6}{\rmdefault}{\mddefault}{\updefault}$(X_1,Y_1)$}}}}
\put(4861,-2356){\makebox(0,0)[lb]{\smash{{\SetFigFont{8}{9.6}{\rmdefault}{\mddefault}{\updefault}$(X_2,Y_2)$}}}}
\put(8821,-2356){\makebox(0,0)[lb]{\smash{{\SetFigFont{8}{9.6}{\rmdefault}{\mddefault}{\updefault}$(X_2,Y_2)$}}}}
\put(7741,-2356){\makebox(0,0)[rb]{\smash{{\SetFigFont{8}{9.6}{\rmdefault}{\mddefault}{\updefault}$(X_2,Y_2)$}}}}
\put(3781,-2356){\makebox(0,0)[rb]{\smash{{\SetFigFont{8}{9.6}{\rmdefault}{\mddefault}{\updefault}$(X_2,Y_2)$}}}}
\put(3826,-2716){\makebox(0,0)[b]{\smash{{\SetFigFont{8}{9.6}{\rmdefault}{\mddefault}{\updefault}$\cT$}}}}
\put(8776,-2716){\makebox(0,0)[b]{\smash{{\SetFigFont{8}{9.6}{\rmdefault}{\mddefault}{\updefault}$\cS$}}}}
\put(3826,-646){\makebox(0,0)[lb]{\smash{{\SetFigFont{8}{9.6}{\rmdefault}{\mddefault}{\updefault}$X_3$}}}}
\put(4726,-646){\makebox(0,0)[lb]{\smash{{\SetFigFont{8}{9.6}{\rmdefault}{\mddefault}{\updefault}$Z_3-X_3$}}}}
\put(6436,-916){\makebox(0,0)[lb]{\smash{{\SetFigFont{8}{9.6}{\rmdefault}{\mddefault}{\updefault}$Z_3-X_3$}}}}
\put(8686,-646){\makebox(0,0)[lb]{\smash{{\SetFigFont{8}{9.6}{\rmdefault}{\mddefault}{\updefault}$Z_3-X_3$}}}}
\put(8326,-646){\makebox(0,0)[lb]{\smash{{\SetFigFont{8}{9.6}{\rmdefault}{\mddefault}{\updefault}$X_3$}}}}
\put(4096,-4336){\makebox(0,0)[lb]{\smash{{\SetFigFont{8}{9.6}{\rmdefault}{\mddefault}{\updefault}$X_3$}}}}
\put(4996,-4336){\makebox(0,0)[lb]{\smash{{\SetFigFont{8}{9.6}{\rmdefault}{\mddefault}{\updefault}$Z_3-X_3$}}}}
\put(8956,-4336){\makebox(0,0)[lb]{\smash{{\SetFigFont{8}{9.6}{\rmdefault}{\mddefault}{\updefault}$Z_3-X_3$}}}}
\put(8596,-4336){\makebox(0,0)[lb]{\smash{{\SetFigFont{8}{9.6}{\rmdefault}{\mddefault}{\updefault}$X_3$}}}}
\put(8281,-4741){\makebox(0,0)[b]{\smash{{\SetFigFont{8}{9.6}{\rmdefault}{\mddefault}{\updefault}$\cS$}}}}
\put(4321,-4786){\makebox(0,0)[b]{\smash{{\SetFigFont{8}{9.6}{\rmdefault}{\mddefault}{\updefault}$\cT$}}}}
\put(6024,-755){\makebox(0,0)[b]{\smash{{\SetFigFont{8}{9.6}{\rmdefault}{\mddefault}{\updefault}$v_3$}}}}
\end{picture}%

%% file: agreement-forest.pdf_t
\begin{picture}(0,0)%
\includegraphics{agreement-forest.pdf}%
\end{picture}%
\setlength{\unitlength}{3522sp}%
\begingroup\makeatletter\ifx\SetFigFont\undefined%
\gdef\SetFigFont#1#2#3#4#5{%
  \reset@font\fontsize{#1}{#2pt}%
  \fontfamily{#3}\fontseries{#4}\fontshape{#5}%
  \selectfont}%
\fi\endgroup%
\begin{picture}(3842,4971)(4380,-1693)
\put(7381,2864){\makebox(0,0)[b]{\smash{{\SetFigFont{8}{9.6}{\rmdefault}{\mddefault}{\updefault}$\cT'$}}}}
\put(6841,2864){\makebox(0,0)[b]{\smash{{\SetFigFont{8}{9.6}{\rmdefault}{\mddefault}{\updefault}$\cT$}}}}
\put(7111,2684){\makebox(0,0)[b]{\smash{{\SetFigFont{8}{9.6}{\rmdefault}{\mddefault}{\updefault}$(X_1,Y_1)$}}}}
\put(7111,2324){\makebox(0,0)[b]{\smash{{\SetFigFont{8}{9.6}{\rmdefault}{\mddefault}{\updefault}$(X_1,Y_1)$}}}}
\put(6031,-241){\makebox(0,0)[b]{\smash{{\SetFigFont{8}{9.6}{\rmdefault}{\mddefault}{\updefault}$\cT$}}}}
\put(7561,-241){\makebox(0,0)[b]{\smash{{\SetFigFont{8}{9.6}{\rmdefault}{\mddefault}{\updefault}$\cT'$}}}}
\put(6031,1469){\makebox(0,0)[b]{\smash{{\SetFigFont{8}{9.6}{\rmdefault}{\mddefault}{\updefault}$\cT$}}}}
\put(7561,1469){\makebox(0,0)[b]{\smash{{\SetFigFont{8}{9.6}{\rmdefault}{\mddefault}{\updefault}$\cT'$}}}}
\put(5806,524){\makebox(0,0)[lb]{\smash{{\SetFigFont{8}{9.6}{\rmdefault}{\mddefault}{\updefault}$X_1$}}}}
\put(7786,524){\makebox(0,0)[rb]{\smash{{\SetFigFont{8}{9.6}{\rmdefault}{\mddefault}{\updefault}$X_1$}}}}
\put(5806,-1186){\makebox(0,0)[lb]{\smash{{\SetFigFont{8}{9.6}{\rmdefault}{\mddefault}{\updefault}$X_1$}}}}
\put(7786,-1186){\makebox(0,0)[rb]{\smash{{\SetFigFont{8}{9.6}{\rmdefault}{\mddefault}{\updefault}$X_1$}}}}
\put(5806,-1546){\makebox(0,0)[lb]{\smash{{\SetFigFont{8}{9.6}{\rmdefault}{\mddefault}{\updefault}$Y_3$}}}}
\put(7786,-1546){\makebox(0,0)[rb]{\smash{{\SetFigFont{8}{9.6}{\rmdefault}{\mddefault}{\updefault}$Y_3$}}}}
\put(5311,-646){\makebox(0,0)[rb]{\smash{{\SetFigFont{8}{9.6}{\rmdefault}{\mddefault}{\updefault}(II)}}}}
\put(5311,1064){\makebox(0,0)[rb]{\smash{{\SetFigFont{8}{9.6}{\rmdefault}{\mddefault}{\updefault}(I) and (III)}}}}
\put(7111,2054){\makebox(0,0)[b]{\smash{{\SetFigFont{8}{9.6}{\rmdefault}{\mddefault}{\updefault}$G_1\cong Q_2$}}}}
\put(7966,2054){\makebox(0,0)[b]{\smash{{\SetFigFont{8}{9.6}{\rmdefault}{\mddefault}{\updefault}$G_2\cong Q_2$}}}}
\put(5266,2279){\makebox(0,0)[b]{\smash{{\SetFigFont{8}{9.6}{\rmdefault}{\mddefault}{\updefault}$(X_3,Y_3)$}}}}
\put(5266,3089){\makebox(0,0)[b]{\smash{{\SetFigFont{8}{9.6}{\rmdefault}{\mddefault}{\updefault}$(X_3,Y_3)$}}}}
\put(5266,1559){\makebox(0,0)[b]{\smash{{\SetFigFont{8}{9.6}{\rmdefault}{\mddefault}{\updefault}$G\cong Q_3$}}}}
\put(5266,1919){\makebox(0,0)[b]{\smash{{\SetFigFont{8}{9.6}{\rmdefault}{\mddefault}{\updefault}$(X_3,Y_3)$}}}}
\put(5266,2729){\makebox(0,0)[b]{\smash{{\SetFigFont{8}{9.6}{\rmdefault}{\mddefault}{\updefault}$(X_3,Y_3)$}}}}
\end{picture}%

%% file: type-1-2.pdf_t
\begin{picture}(0,0)%
\includegraphics{type-1-2.pdf}%
\end{picture}%
\setlength{\unitlength}{3522sp}%
\begingroup\makeatletter\ifx\SetFigFont\undefined%
\gdef\SetFigFont#1#2#3#4#5{%
  \reset@font\fontsize{#1}{#2pt}%
  \fontfamily{#3}\fontseries{#4}\fontshape{#5}%
  \selectfont}%
\fi\endgroup%
\begin{picture}(5424,1507)(4399,-3536)
\put(4501,-3256){\makebox(0,0)[b]{\smash{{\SetFigFont{8}{9.6}{\rmdefault}{\mddefault}{\updefault}$W$}}}}
\put(4861,-3256){\makebox(0,0)[b]{\smash{{\SetFigFont{8}{9.6}{\rmdefault}{\mddefault}{\updefault}$X_1$}}}}
\put(4636,-2761){\makebox(0,0)[rb]{\smash{{\SetFigFont{8}{9.6}{\rmdefault}{\mddefault}{\updefault}$v$}}}}
\put(4681,-3481){\makebox(0,0)[b]{\smash{{\SetFigFont{8}{9.6}{\rmdefault}{\mddefault}{\updefault}$\cT$}}}}
\put(8461,-3256){\makebox(0,0)[b]{\smash{{\SetFigFont{8}{9.6}{\rmdefault}{\mddefault}{\updefault}$W$}}}}
\put(7021,-3256){\makebox(0,0)[b]{\smash{{\SetFigFont{8}{9.6}{\rmdefault}{\mddefault}{\updefault}$W$}}}}
\put(7381,-3256){\makebox(0,0)[b]{\smash{{\SetFigFont{8}{9.6}{\rmdefault}{\mddefault}{\updefault}$X_1$}}}}
\put(7156,-2761){\makebox(0,0)[rb]{\smash{{\SetFigFont{8}{9.6}{\rmdefault}{\mddefault}{\updefault}$v$}}}}
\put(7381,-3481){\makebox(0,0)[b]{\smash{{\SetFigFont{8}{9.6}{\rmdefault}{\mddefault}{\updefault}$\cT$}}}}
\put(9721,-3256){\makebox(0,0)[b]{\smash{{\SetFigFont{8}{9.6}{\rmdefault}{\mddefault}{\updefault}$X_1$}}}}
\put(5851,-3256){\makebox(0,0)[b]{\smash{{\SetFigFont{8}{9.6}{\rmdefault}{\mddefault}{\updefault}$X_1$}}}}
\put(6211,-3256){\makebox(0,0)[b]{\smash{{\SetFigFont{8}{9.6}{\rmdefault}{\mddefault}{\updefault}$Z$}}}}
\put(6031,-3481){\makebox(0,0)[b]{\smash{{\SetFigFont{8}{9.6}{\rmdefault}{\mddefault}{\updefault}$\cT'$}}}}
\put(6076,-2761){\makebox(0,0)[lb]{\smash{{\SetFigFont{8}{9.6}{\rmdefault}{\mddefault}{\updefault}$v'$}}}}
\put(9046,-2311){\makebox(0,0)[rb]{\smash{{\SetFigFont{8}{9.6}{\rmdefault}{\mddefault}{\updefault}$v'$}}}}
\put(9361,-3256){\makebox(0,0)[b]{\smash{{\SetFigFont{8}{9.6}{\rmdefault}{\mddefault}{\updefault}$Z$}}}}
\put(9091,-3481){\makebox(0,0)[b]{\smash{{\SetFigFont{8}{9.6}{\rmdefault}{\mddefault}{\updefault}$\cT'$}}}}
\put(8821,-2536){\makebox(0,0)[rb]{\smash{{\SetFigFont{8}{9.6}{\rmdefault}{\mddefault}{\updefault}$g$}}}}
\end{picture}%